\documentclass[10pt,twoside]{article}

\usepackage{fancyhdr}
\newlength{\myparskip}              \newlength{\myparindent}             %
\newlength{\oldparskip}             \newlength{\oldparindent}            %
\parskip\myparskip                  \parindent\myparindent               %
\fancypagestyle{plain}{%
   \fancyhead[LO]{}
   \fancyhead[CO]{}\fancyhead[RO]{}
}

\newenvironment{proof}{\topsep=\smallskipamount \partopsep=0pt  %
 \begin{trivlist} \itemindent=\parindent                        %
 \item[\hskip \labelsep\emph{Proof:}]}{\qed\end{trivlist}}      %
\let\qed=\relax                                                 %
\def\qed                                                        %
 {{\unskip\nobreak\hfil\penalty50                               %
   \quad\hbox{}\nobreak\hfil $\Box$                             %
   \parfillskip=0pt \finalhyphendemerits=0 \par}}               %

\def\@thmcountersep{-}                                          %
\newtheorem{theorem}{Theorem}[section]                          %
\newtheorem{proposition}[theorem]{Proposition}                  %
\newtheorem{lemma}[theorem]{Lemma}                              %

\newcommand{\institute}[1]{\newcommand{\where}{#1}}

\renewcommand{\and}{\text{and }}

\newcommand{\letitre}{The Nahm transform for calorons}
\newcommand{\ladate}{May 16, 2007.}
\title{\letitre}
\author{Benoit Charbonneau
   \and Jacques Hurtubise
\thanks{The first author is supported by an NSERC PDF. The second author is
supported by NSERC and FQRNT grants.  The diagrams in this paper
were created using Paul Taylor's Commutative Diagram package.}}
\institute{Department of Mathematics and Statistics, McGill University\\
           805 Sherbrooke St. W, Montreal, Quebec, H3A 2K6, Canada.\\
           E-mails: benoit@alum.mit.edu and jacques.hurtubise@mcgill.ca}
\date{\ladate}
\usepackage{amssymb,amsmath}
\usepackage[size=1.8em,PS]{diagrams}  
\usepackage[hyperindex=true,
pdfauthor={Benoit Charbonneau (benoit@alum.mit.edu) and
Jacques Hurtubise (jacques.hurtubise@mcgill.ca)},pdftitle={\letitre},
pdfsubject={\ladate} ]{hyperref}
\DeclareMathAlphabet{\mathdj}{U}{msb}{m}{n}
\newcommand{\R}{\ensuremath{\mathdj {R}}}  
\newcommand{\C}{\ensuremath{\mathdj {C}}}  
\newcommand{\Z}{\ensuremath{\mathdj {Z}}}  
\newcommand{\PP}{\ensuremath{\mathdj {P}}} 
\newcommand{\PPP}{{\PP^1\times\PP^1}}      
\newcommand{\bbr}{\R}\newcommand{\bbz}{\Z}  
\newcommand{\bbc}{\C}\newcommand{\bbp}{\PP} 
\newcommand\bb[1]{{\text{\bf#1}}}           
\newcommand{\bbi}{\bb{I}}                   
\newcommand\bbt{\bb{T}}                     
\newcommand\func[1]{\operatorname{\mathrm{#1}}}
\newcommand\im{\func{im}}                   
\newcommand{\SR}{S^1\times\R^3}
\newcommand{\OO}{{\mathcal{O}}}            
\renewcommand{\Im}{\mathrm{Im}}            

\def\citep#1#2{\cite[{#1}]{#2}}

\newcount\notenumber \notenumber=1
\newcommand\note[1]%
{$^{\the\notenumber}$\marginpar{\sf\footnotesize{$^{\the\notenumber}${#1}}}
\global\advance\notenumber by 1}

\begin{document}

\maketitle
\begin{center}\emph{Dedicated to Nigel Hitchin
         on the occasion of his sixtieth birthday.}\end{center}
\begin{abstract}
In this paper, we complete the proof of an equivalence given by Nye
and Singer of the equivalence between calorons (instantons on
$S^1\times \R^3$) and
solutions to Nahm's equations over the circle, both satisfying appropriate
boundary conditions.  Many of the key ingredients are
provided by a third way of encoding the same data which involves
twistors and complex geometry.
\end{abstract}

\section{Introduction}
One rather mysterious feature of the self-duality equations on
$\bbr^4$ is the existence of a quite remarkable non-linear transform, the Nahm
transform.  It maps solutions to the self-duality equations
on $\bbr^4$ invariant under a closed translation group $G\subset \bbr^4$ to
solutions to the self-duality equations on $(\bbr^4)^*$ invariant
under the dual group $G^*$. This transform uses spaces of solutions to
the Dirac equation, it is quite sensitive to boundary
conditions, which must be defined with care, and it is not straightforward: for
example it tends to interchange rank and degree.

The transform was introduced by Nahm as an adaptation of
the original ADHM construction of instantons (\cite{ADHM}) having the advantage
that it can be generalised.  The series of papers
\cite{Nahm,nahmEqn,corrigan-goddard} details Nahm's original transform.
In the case monopole case $G=\R$, the
transform takes the monopole to a solution to some
non-linear matrix valued o.d.e.s, Nahm's equations on an interval,
and there is an inverse transform giving back the monopole. The
mathematical development of this transform is due to Nigel
Hitchin(\cite{hitchinMonopoles}), who showed for the $SU(2)$ monopole how
the monopole and the corresponding solution to Nahm's equations are
both encoded quite remarkably in the same algebraic curve, the
spectral curve of the monopole. This work was extended to the cases of
monopoles for the other classical groups  in
\cite{hurtubiseMurray}. This extension illustrates just how odd
the transform can be: one gets solutions to the Nahm equations on a
chain of intervals, but the size of the matrices jump from interval to
interval.

The Nahm transform for other cases of $G$-invariance has been studied
by various authors (among others \cite{braam1989,benoitthesis,cherkis2001,
jardim2001,jardim2002,jardim2002b,nye,schenk1988,szilard});
a nice survey can be found in \cite{jardimsurvey}. Here,
we study the case of  ``minimal'' invariance,
  under $\bbz$; the fields are referred to as calorons, and the Nahm
  transform sends them to solutions of Nahm's equations over the
  circle. This case is very close to the monopole one, and indeed
  calorons can be considered as Ka\v c--Moody monopoles
  (\cite{garland-murray}).

This project started while we were studying the work of Nye and Singer
(\cite{nye,singernye}) on calorons; they consider the Nahm transform
directly, and do most of the work required to show that the transform
is involutive. The missing ingredients turn out to lie in complex
geometry, in precisely the same way Hitchin's extra ingredient of a
spectral curve complements Nahm's. The complex geometry  allows us
to complete the equivalence, which can then be used to compute the
moduli. It thus seems to us quite appropriate to consider this problem
in a volume dedicated to Nigel Hitchin, as it allows us to revisit
some of his beautiful mathematics, including some on calorons which
has remained unpublished.

In Section \ref{sec:NyeSinger}, we summarise the work of Nye and Singer
towards showing that the Nahm transform is an equivalence
between calorons and appropriate solutions to Nahm's equations.
In Section \ref{sec:Kac-Moody}, following in large part on work
of Garland and Murray, we describe the complex geometry (``spectral
data'') that encodes a caloron. In Section \ref{sec:NahmSpectralData},
we study the process by
which spectral data also correspond to solutions to
Nahm's equations. In Section \ref{sec:closingcircle}, we close the circle, showing the
two Nahm transforms are inverses.  In Section \ref{moduli}, we give a
description of moduli, expounded in \cite{benoitjacques1}.


\section{The work of Nye and Singer}\label{sec:NyeSinger}
\subsection{\texorpdfstring{Two types of invariant self-dual gauge fields on $\bbr^4$.}%
                           {Two types of invariant self-dual gauge fields on R4.}}
\label{sec:gaugefields}
Nye and Singer study the Nahm transform between the following two
self-dual gauge fields (we restrict ourselves to $SU(2)$, thought they
study the more general case of $SU(N)$):
\bigskip

\noindent{\bf A- $SU(2)$ Calorons of charge $(k,j)$.}

$SU(2)$-calorons are \emph{self-dual} $SU(2)$-connections on
$\SR$, satisfying appropriate boundary conditions. We view $\SR$ as
the quotient of the standard Euclidean $\bbr^4$ by the time
translation $(t,x)\mapsto (t+2\pi/\mu_0,x)$.
Let $A$ be such a connection, defined over a rank 2 vector bundle $V$
equipped with a unitary structure; we write it in coordinates over
$\bbr^4$ as
\begin{equation*}
A = \phi dt+\sum_{i=1,2,3}A_idx_i,
\end{equation*}
with associated covariant derivatives
\begin{equation*}
\nabla = \bigl(\frac{\partial}{\partial t} + \phi\bigr) dt+
\sum_{i=1,2,3}\bigl(\frac{\partial}{\partial x_i} + A_i\bigr) dx_i
 \equiv \nabla_t dt+\sum_{i=1,2,3}\nabla_i dx_i.
\end{equation*}
We require that the $L^2$ norm of the curvature of $A$ be
finite, and that in suitable gauges, the $A_i$ be $O(|x|^{-2})$, and that
$\phi$ be conjugate to $diag (i(\mu_1 - j/2|x|), i(-\mu_1 +j/2|x|)) +
O(|x|^{-2})$ for a positive real constant  $\mu_1$ and a positive
integer  $j$ (the monopole charge). We also have bounds on the derivatives
of these fields.

The boundary conditions tell us in essence that in a suitable way the
connection extends to the $2$-sphere at infinity times $S^1$;
furthermore, one can show that the extension is to a fixed connection,
which involves fixing a trivialisation at infinity;
there is thus a second invariant we can define, the relative second Chern
class, which we represent by a (positive) integer $k$. There are then
two integer charges for our caloron, $k$ and $j$, respectively the
instanton and monopole charges.

There is  a suitable group of gauge transformations acting on these
fields  \cite{nye}: Nye's approach is to compactify to $S^1\times \bar B^3$
with a fixed trivilisation over the boundary $S^1\times S^2$. The gauge
transformations are those extending smoothly the identity on the boundary.
Two solutions are considered to be equivalent if they are in the same
orbit under this group.

\bigskip

{\bf B- Solutions to Nahm's equations on the circle.}

The second class of objects we consider are skew adjoint matrix valued
functions $T_i(z), i=0,\ldots,3$, of size
$(k+j)\times(k+j)$ over the interval $(-\mu_1, \mu_1)$, and of size
$(k)\times(k)$ over the interval $(\mu_1,\mu_0-\mu_1)$,
(hence $\mu_1<\mu_0-\mu_1$, so we impose that condition), that are
solutions to Nahm's equations
\begin{equation}\label{eqn:Nahm}
\frac{dT_{\sigma(1)}}{dz} + [T_0,T_{\sigma(1)}] =[T_{\sigma(2)},T_{\sigma(3)}],
\quad \text{ for $\sigma$ even permutations of }(123),
\end{equation}
on the circle $\bbr/(z\mapsto z+\mu_0)$.
These  equations are reductions of the self-duality
equations to one dimension, and are invariant under a group of gauge
transformations under which the
$\frac{d}{dz} + T_0$ transforms as a connection.

We think of the $T_i$ as sections of the endomorphisms of a vector bundle
$K$ whose rank jumps at the two boundary points; at these points $\pm
\mu_1$ we  need boundary conditions.
\bigskip

{\it Case 1: $j\neq 0$}.  At each of the boundary points, there is a
large side, with a rank $(k+j)$ bundle, and a small side, with a rank
$k$ bundle.  We attach the two at the boundary point using an
injection $\iota$ from the small side into the large side, and a
surjection $\pi$ going the other way, with $\pi\cdot \iota$ the identity.
One asks, from the small side, that the $T_i$ have well defined limits
at the boundary point. From the large side, one has a decomposition
near the boundary points of the bundle $\bbc^{k+j}\times [-\mu_1,
\mu_1]$ into an orthogonal sum of subbundles of
rank $k$ and $j$  invariant with respect to the connection $\frac{d}{dz}
+ T_0$ and compatible with the maps $\iota, \pi$. With respect to
this decomposition the $T_i$ have the form
\begin{equation}
T_i = \begin{pmatrix} \hat T_i& O(\hat z^{(j-1)/2})\\
                    O(\hat z^{(j-1)/2})&R_i\end{pmatrix},
\end{equation}
in a basis where $T_0$ is gauged to zero and for a choice of coordinate
 $\hat z := z-(\pm\mu_1)$.  At the
boundary point  $\hat z = 0$, each $\hat T_i$ has a well defined limit
that coincides using $\pi, \iota$ with the limit from the small side. The
$R_i$ have simple poles, and we ask that their residues form an
irreducible representation of $su(2)$.
\bigskip

{\it Case 2: $j=0$}.
Here the boundary conditions are different. We have at the boundary,
identification of the fibres of $K$ from both sides. With this
identification, we ask  that, in a gauge where $T_0$ is zero, the
$T_i$ have well defined limits $(T_i)_\pm$ from both sides,
and that, setting
\begin{equation}\label{eqn:A}
A(\zeta) = (T_1+ iT_2) -2iT_3\zeta +
  (T_1-iT_2)\zeta^2, \end{equation}
one has that taking the limits from both sides,
\begin{equation}\label{jump}A(\zeta)_+ - A(\zeta)_- =
  (\alpha_0+\alpha_1\zeta)(\bar \alpha_1- \bar \alpha_0\zeta)^T
\end{equation}
for vectors $\alpha_0,\alpha_1$ in $\bbc^k$. In particular,
$A(\zeta)_+ - A(\zeta)_-$ is of rank at most one for all $\zeta$.
\bigskip

In both cases, there is a symmetry condition, that, in a suitable gauge,
\begin{equation*}
T_i(0) = T_i(0)^T
\end{equation*}
and, finally, an irreducibility condition, that there be no covariant
constant sections of the bundle left invariant under the matrices
$T_i$. Solutions for which there are such covariant constant sections
 should correspond to cases where instanton charge has bubbled off,
leaving behind a caloron of lower charge.

\subsection{The Nahm transform}
\subsubsection*{From caloron to a solution to Nahm's equations.}
Let $z$ be a real parameter. Using the Pauli matrices $e_i$,
we define  the Dirac operators $D_z$
acting on sections of the tensor product of
the vector bundle $V$
with the spin bundles
$S_\pm \simeq \bbc^2\times \SR$ by
\begin{align*}
D_z\colon \Gamma(V\otimes S_+)&\rightarrow \Gamma(V\otimes S_-)\\
  s &\mapsto ( \nabla_t + iz)s + \sum_i (e_i\nabla_i ) s.
\end{align*}
The Weitzenbock formula $D_zD_z^*=\nabla^*\nabla+\rho(F_{A_z}^-)$
guarantees that $D_z^*$ is injective.
Nye and Singer show in \cite{singernye} that for  $z$ not in the set of lifts
$\pm\mu_1+ n \mu_0$ ($n\in \bbz$) of the
boundary point $\pm\mu_1$ from the circle to $\bbr$, and
with suitable choices of function
spaces, the operator $D_z$ is Fredholm, and its
index away from these points is $k+j$ for $z$ lying in
the  intervals $(-\mu_1 + n\mu_0,\mu_1 +n\mu_0)$ and is $k$ in the
intervals $(\mu_1 + n\mu_0,-\mu_1 +(n+1)\mu_0)$. There is thus a bundle
over $\bbr$ whose rank jumps at $\pm\mu_1+ n \mu_0$ ($n\in \bbz$), with
fibre $\ker (D_z)$ at $z$. There is a natural way of
shifting by $\mu_0$, identifying $\ker (D_z)$ with $\ker
(D_{z+\mu_0})$, giving a bundle $K$ over the circle.

Over each interval, this bundle sits inside the trivial bundle whose
fibre is the space of $L^2$ sections of $V\otimes S_-$. Let $P$ be the
orthogonal projection from this trivial bundle onto $K$.
As elements of $K$ decay exponentially,
the operation $X_i$ of multiplying by the coordinate $x_i$ can be
used to define operators on sections of $K$ by
\begin{align*}
 \frac{d}{dz} + T_0&= P\cdot \frac{d}{dz},\\
T_i& = P\cdot X_i.
\end{align*}

\begin{theorem}[\citep{Sec 4.1.2, p.108}{nye}]
The operators defined in this way satisfy Nahm's equations.
\end{theorem}

This theorem and Theorem \ref{nahm-to-caloron} below
fall in line with the general Nahm transform heuristic
philosophy: the curvature
of the transformed object, seen as an invariant connection on $\R^4$,
is always composed of a self-dual piece and another piece depending on the
behaviour at infinity of the harmonic spinors.  Since these are decaying
exponentially here, that other piece is zero and
the transformed object is self-dual, or equivalently
once we reduce, it satisfies Nahm's equations; see for instance
\cite[Sec. 3]{benoitpaper}.

\subsubsection*{From  a solution to Nahm's equations to a caloron.}
For the inverse transform, we proceed the same way: from a solution to Nahm's equations,
we define a family of auxiliary
operators parameterised by points of $\bbr^4$
acting on sections of $K\otimes \bbc^2$ by
\begin{equation*}
D_{x,t} = i\bigl(\frac{d}{dz} +T_0 -it + \sum_{j= 1,2,3}e_j (T_j -ix_j)\bigr).
\end{equation*}

 One must again distinguish two cases, $j>0$ and $j=0$.

{ $j>0$:}  We define the  space $W$ of $L^2_1$ sections of $K\otimes
\bbc^2$ such that at $\pm\mu_1$ the values of the sections coincide:
for the sections $s_1$ on the small side and
$s_2$ on the large side, we need $\iota(s_1)=s_2$.
Let $X$ be the space of $L^2$ sections of $K\otimes \bbc^2$ over the
circle. Set
\begin{equation}
\hat D_{x,t} := D_{x,t}\colon W\rightarrow X
\end{equation}

{ $j=0$:} We have as above the space $W$ of sections. We define in addition
a two dimensional space $U$ associated with the end points
$\pm\mu_1$. The jump condition given by Equation (\ref{jump}) at $\mu_1$
imply that there is a vector $u_+ \in V_{\mu_1}\otimes\bbc^2$
such that, as elements of $End(V_{\mu_1}\times \bbc^2)$,
\begin{equation*}\sum_{j=1}^3 \bigl((T_j)_+-(T_j)_-\bigr)\otimes e_j =
  (u_+ \otimes u_+ ^*)_0.
\end{equation*}
Here the subscript $(\ )_0$ signifies
taking the trace free $End(V)\otimes Sl(2,\bbc)$
component inside $End(V)\otimes End(\bbc^2)\simeq End(V\otimes\bbc^2)$.

One has a similar vector $u_-$ at $-\mu_1$. Let $U$ be the
vector space spanned by $u_+, u_-$. Let $\Pi\colon V_{\mu_1}\oplus V_{-\mu_1}\rightarrow U$
be the orthogonal projection, let $X$
be the sum of the space of $L^2$ sections of V  with the space $U$,
and set
\begin{equation}
\hat D_{x,t} := (D_{x,t}, \Pi)\colon W\rightarrow X
\end{equation}
The kernel of this operator consists of sections $s$ in the kernel of
$D_{x,t}$, with values at $\pm\mu_1$ lying in $U^\perp$; the cokernel
consists of triples $(s,c_{\mu_1}u_{\mu_1},c_{-\mu_1}u_{-\mu_1})$,
where $s$ is a  section of $K\otimes \bbc^2$, lying in the kernel of
$D_{x,t}^*$, with jump discontinuity $c_+u_+$ at $\mu_1$,
and $c_-u_-$ at $-\mu_1$.

\begin{theorem}[\citep{p. 104}{nye}] \label{nahm-to-caloron}
The operator $\hat D_{x,t}$ has trivial kernel for all $(x,t)$, and
has a cokernel of rank 2, defining a rank 2 vector bundle over
$\bbr^4$, with natural time periodicity which allows one to build a
vector bundle $V$ over $\SR$, defined locally as a
subbundle of the infinite dimensional bundle $X\times \SR$.
Let $P$ denote the orthogonal projection from $X$ to
$V$. Setting, on sections of $V$,
\begin{align*}
\nabla_i &= P\cdot \frac{\partial}{\partial x_i}, \quad i =1,2,3\\
\nabla_t &= P\cdot \frac{\partial}{\partial t},
\end{align*}
defines an $SU(2)$-caloron.
\end{theorem}

\subsection{Involutivity of the transforms.}
We would like these two transforms to be inverses of each other, so
that they define an equivalence. The results of Nye and Singer get us
most of the way there.  What is missing is a proof that the solutions
to Nahm's equations one obtains from a caloron satisfy the correct
boundary and irreducibility conditions and then, that the two
constructions are inverses to one another.

The easiest way to do so, as Hitchin did in his original paper on
$SU(2)$ monopoles, is to exploit the regularity given by a third
equivalent set of data, which involves complex geometry. To do this,
we examine some  work of
Garland and Murray, building on some  unpublished work of Hitchin.


\section{\texorpdfstring{The twistor transform for calorons/Ka\v c--Moody monopoles}%
                        {The twistor transform for calorons/Kac-Moody monopoles}}
\label{sec:Kac-Moody}
\subsection{Upstairs: the twistor transform for calorons.}
Like all self-dual gauge fields on $\bbr^4$ or its quotients
$\bbr^4/G$, calorons admit a twistor transform, translating the gauge
fields into holomorphic vector bundles on an auxiliary space, the
twistor space associated to  $\bbr^4/G$. Following
\cite[Sec. 2]{garland-murray}, we summarise the construction, again for
$SU(2)$ only.

It is convenient first to recall from \cite[Sec. 3]{hitchinGeodesics} the twistor space
for $\bbr^3$. It can be interpreted as the space of oriented lines
in $\bbr^3$ and it is isomorphic to $T\bbp^1$, the tangent bundle of the
Riemann sphere. Let $\zeta$ be the standard affine coordinate on
$\bbp^1$, and $\eta$ be the fibre coordinate in $T\bbp^1$ associated
to the basis vector ${\partial}/{\partial \zeta}$. The
incidence relation between the standard coordinates on $\bbr^3$ and
$(\zeta, \eta)$ is given by
\begin{equation*}
\eta = (ix_1-x_2)+2x_3\zeta + (ix_1+x_2)\zeta^2.
\end{equation*}
The space $T\bbp^1$ comes equipped with standard line bundles ${\cal
  O}(k)$, lifted from $\bbp^1$, and line bundles $L^t$, parameterised by
$t\in \bbc$,
with transition function $exp (t\eta/\zeta)$ from the open set
$\zeta\neq\infty$ to the open set $\zeta\neq 0$.
For $t\in\R$, these line bundles correspond to the standard $U(1)$
monopoles on $\bbr^3$ given by a flat connection and constant Higgs
field $it$. Let $L^t(k) := L^t\otimes {\cal O}(k)$. These bundles
are in some sense the building blocks for monopoles for higher gauge
groups.

According to Garland and Murray,  the twistor space $\cal T$ for $\SR$
parametrises pairs consisting of a
point in $S^1\times \R^3$ and a unit vector in $\R^3$.  Such a pair gives
a cylinder along which to integrate, and that projects to
an oriented line in $\R^3$.  There is then a $\bbc^*$
fibration $\pi\colon {\cal T}\rightarrow
T\bbp^1$, which is in fact holomorphic and
the complement of the zero
section in $L^{\mu_0}$. It has a natural fibrewise compactification
$\tilde{\cal T} = \bbp({\cal O}\oplus L^{\mu_0})$
compactifying the cylinders into spheres. The complement $\tilde{\cal
  T}\backslash{\cal T}$ is a sum of two disjoint divisors ${\cal T}^0$
and ${\cal T}^\infty$ mapping isomorphically to $T\bbp^1$.

In the monopole case, the holomorphic vector bundle was obtained by
integrating $\nabla_s - i\phi$ along real lines using a metric
coordinate $s$ along the line, and the Higgs field $\phi$. Here, the
analogous operation is integrating $\nabla_s - i\nabla_t$ over the
cylinders to obtain a vector bundle $E$ over $\cal T$. The
boundary conditions allow us to extend the bundle to the
compactification $\tilde{\cal T}$; we denote this extension again by
$E$. The boundary conditions also give line subbundles over these
divisors, $E^0_1 =L^{-\mu_1}(-j)$ over
${\cal T}^0$, corresponding  to solutions of $\nabla_s - i\phi$
which decay as one goes to plus infinity on the cylinder, and
$E^\infty_1 =L^{\mu_1}(-j)$ over
${\cal T}^\infty$, corresponding to solutions which decay as one goes
to minus infinity.

The twistor space has a real structure $\tau$, which acts on the
bundle by $\tau^*(E) = E^*$, and maps the line subbundle $L^{-\mu_1}(-j)$
over ${\cal T}^0$ to the annihilator of the subbundle $L^{\mu_1}(-j)$ over
${\cal T}^\infty$.

\subsection{\texorpdfstring{Downstairs: the caloron as a Ka\v c--Moody monopole.}
                           {Downstairs: the caloron as a Kac-Moody monopole.}}
As Garland and Murray show in \cite[Sec. 6]{garland-murray},
there is very nice way of
thinking of the caloron as a monopole over $\bbr^3$, with values in a
Ka\v c--Moody algebra, extending a loop algebra: the fourth variable $t$
becomes the internal loop algebra variable. This way of thinking goes over
to the twistor space picture.

Indeed, the twistor transform for the caloron gives us a bundle $E$
over ${\cal T}$. Taking a direct image (restricting to sections with
poles of finite order along $0,\infty$) gives one an infinite
dimensional vector bundle $F$ over $T\bbp^1$. If $w$ is a standard
fibre coordinate on ${\cal T}$ vanishing over ${\cal T}^0$ and with a
simple pole at ${\cal T}^\infty$, it induces an endomorphism $W$ of
$F$, and quotienting $F$ by the $\cal O$-module generated by the image
of $W-w_0$ gives us the restriction of $E$ to the section $w=w_0$, so
that $E$ and $(F,W)$ are equivalent. One has more, however; the fact
that the bundle $E$ extends to $\tilde {\cal T}$ gives a subbundle
$F^0$ of sections in the direct image which extend over ${\cal T}^0$,
and a subbundle $F^\infty$ of sections in the direct image which
extend over ${\cal T}^\infty$. One can go further and use the
flags $0=E^0_0\subset E^0_1\subset E^0_2 =E$ over  ${\cal T}^0$,
$0=E^\infty_0\subset E^\infty_1\subset E^\infty_2 =E$ over  ${\cal
  T}^\infty$ to define for $p\in \bbz$ and $q= 0,1$ subbundles $F^0_{p,q},
F^\infty_{p,q}$ of $F$ as
\begin{align*}
F^0_{-p,q} &=\{s\in F\mid w^{-p}s \text{ finite at }{\cal T}^0
                         \text{ with value in } E^0_q  \}, \\
F^\infty_{p,q} &=\{s\in F\mid w^{-p}s \text{ finite at } {\cal T}^\infty
                           \text{ with value in } E^\infty_q \}.
\end{align*}
We now have infinite flags
\begin{equation}\label{infiniteflags}\begin{gathered}
\cdots\subset F^0_{-1,0}\subset F^0_{-1,1} \subset F^0_{0,0} \subset F^0_{0,1} \subset F^0_{1,0} \subset F^0_{1,1} \subset\cdots\phantom{-.}\\
\cdots\supset F^\infty_{2,0}\supset F^\infty_{1,1} \supset F^\infty_{1,0} \supset F^\infty_{0,1} \supset F^\infty_{0,0} \supset F^\infty_{-1,1} \supset\cdots.
\end{gathered}\end{equation}
Note that $W\cdot F^0_{p,q} = F^0_{p+1,q}$ and $W\cdot F^\infty_{p,q}=
F^\infty_{p-1,q}$.  Garland and Murray show:
\begin{itemize}
\item $F^0_{p,0}$ and $F^\infty_{-p+1,0}$ have zero intersection and sum to
  $F$ away from a compact curve $S_0$ lying in the linear system
  $|{\cal O}(2k)|$;
\item $F^0_{p,1}$ and $F^\infty_{-p,1}$ have zero intersection and sum to
  $F$ away from a compact curve $S_1$ lying in the linear system
  $|{\cal O}(2k+2j)|$;
\item the quotients $F^0_{p,0}/F^0_{p-1,1}$ are line bundles
  isomorphic to $L^{(p-1)\mu_0+\mu_1}(j)$;
\item the quotients $F^0_{p,1}/F^0_{p,0}$ are line bundles isomorphic
  to $L^{p\mu_0-\mu_1}(-j)$;
\item the quotients $F^\infty_{p,0}/F^\infty_{p-1,1}$ are line
  bundles isomorphic to $L^{-(p-1)\mu_0-\mu_1}(j)$;
\item the quotients $F^\infty_{p,1}/F^\infty_{p,0}$ are line bundles
  isomorphic to $L^{-p\mu_0+\mu_1}(-j)$.
\end{itemize}

It is worthwhile stepping back now and seeing what we have from a
group theoretic point of view. On $T\bbp^1$, we have a bundle with
structure group  $G=\widetilde{Gl}(2,\bbc)$
of $Gl(2,\bbc)$-valued loops. The flags that we have found give two reductions $R_0,
R_\infty$ to the opposite Borel subgroups $B_0$, $B_\infty$ (of loops
extending to $0, \infty$ respectively in $\bbp^1$, and preserving
flags over these points) in
$G$. As for finite dimensional groups, we have exact sequences
relating the Borel subgroups, their unipotent subgroups, and the
maximal torus:
\begin{equation}
\begin{matrix}
0&\rightarrow&U_0&\rightarrow&B_0&\rightarrow&T&\rightarrow&0,\\
0&\rightarrow&U_\infty&\rightarrow&B_\infty&\rightarrow&T&\rightarrow&0.\\
\end{matrix}
\end{equation}
In a suitable basis, $B_0$, $B_\infty$ consist respectively of upper and lower
triangular matrices, and $T$ of diagonal matrices.

The two  reductions  are generically transverse, and fail to be
transverse over the spectral curves of the caloron. This failure of
transversality gives us geometrical data which encodes the bundle and
hence the caloron. This approach is expounded in
\cite{hurtubiseMurraySpectral}
for monopoles, but can be extended to calorons
in a straightforward fashion.
We summarise the construction here. The
vector bundle $F$ can be thought of as defining an element $f$ of the
cohomology set $H^1(T\bbp^1, G)$, and so a principal bundle $P_f$; one
simply thinks of $f$ as the transition functions for $P_f$, in terms
of \v Cech cohomology. The reductions to  $B_0$, $B_\infty$, can also be
thought of as elements $f_0$, $f_\infty$ of
$H^1(T\bbp^1, B_0)$, $H^1(T\bbp^1, B_\infty)$, respectively, or as
principal bundles
$P_{f_0}$, $P_{f_\infty}$, or, since they are reductions of $P_f$, as
sections $R_0$ of $P_f\times_G G/B_0 =
P_{f_\infty}\times_{B_\infty}G/B_0$ and $R_\infty$ of $P\times_G
G/B_\infty = P_{f_0}\times_{B_0}G/B_\infty$.

We note also that the elements $f_0$, $f_\infty$ when projected to
$H^1(T\bbp^1, T)$, give fixed elements $\alpha_0$, $\alpha_\infty$, in
the sense that they are independent of the caloron, or rather depend
only on the charges and the asymptotics, which we presume fixed. It
is a consequence of the fact that the successive quotients
in Sequence (\ref{infiniteflags})  depend only on the
$\mu_i$ and the charges. Let $A_0,A_\infty$ denote the principal $T$
bundles associated to $\alpha_0$, $\alpha_\infty$.

Let $A_0(U_0)$ be the sheaf of sections of the $U_0$ bundle associated to $A_0$ by
the action of $T$ on $U_0$.
The cohomology set $H^1(T\bbp^1,A_0(U_0))$ describes (\cite[Sec. 3]{hurtubiseMurraySpectral})
the set of $B_0$ bundles mapping to $\alpha_0$.
To find the bundle $f_0$, and so $f$, then, one simply needs
to have the appropriate class in $H^1(T\bbp^1,A_0(U_0))$.

Note that $U_0$ also serves as the big cell in the homogeneous
space $G/B_\infty$. Following \cite{Gravesen}, we
 think of $G/B_\infty$ as adding an infinity to $U_0$, so if
${\cal U}_0$ denotes the sheaf of holomorphic maps into $U_0$,  the
sheaf $\cal M$ of holomorphic maps into $G/B_\infty$ can be thought of
as meromorphic maps into $U_0$. Hence there is an
sequence of sheaves of pointed sets (base points are chosen compatibly
in both $U_0$ and $G/B_\infty$)
\begin{equation*}
0\rightarrow {\cal U}_0 \rightarrow{\cal M}\rightarrow {\cal P}r\rightarrow 0
\end{equation*}
defining the sheaf ${\cal P}r$ of principal parts. As $B_0, T$
act on these sheaves, we can build the  twisted versions
\[\begin{diagram}
P_{f_0}({\cal U}_0) &\rTo &P_{f_0}({\cal M}) &\rTo^{\phi}&P_{f_0}(\mathcal{P}r),\\
A_0({\cal U}_0) &\rTo &A_0({\cal M})& \rTo &A_0({\cal P}r).
\end{diagram}\]
As we are quotienting out the action of $U_0$,
we have $A_0({\cal P}r)\simeq P_{f_0}({\cal P}r)$.

We are now ready to define the \emph{principal part data} of the caloron.
The principal part data of the caloron is the image under $\phi$ of the class of
the reduction given as a section
$R_\infty$ of $P_{f_0}\times_{B_0}G/B_\infty = P_{f_0}({\cal M})$:
\begin{equation*}
\phi(R_\infty) \in H^0(A_0({\cal P}r)).
\end{equation*}
To get back the caloron bundle from the principal part data, one takes the
coboundary $\delta(\phi(R_\infty))$, in the obvious \v Cech sense, for
the second sequence, obtaining a class in $H^1(T\bbp^1, A_0({\cal
  U}_0))$. One checks that this is precisely the class corresponding to
the bundle $P_0$; while this seems a bit surprising, the construction
is fairly tautological, and is done in detail in
\cite[Sec. 3]{hurtubiseMurraySpectral}. In our infinite dimensional context,
the bundles are quite special, as they, and their reductions, are
invariant under the shift operator $W$.

Of course, this construction does not tell us what the mysterious class
$\phi(R_\infty)$ actually corresponds to, but it turns out that, over
a generic set of calorons, it is quite tractable. Indeed, the principal part
data, as a sheaf, is supported over the spectral curves, and, if the
curves have no common components and no multiple components, then the
principal part data amounts to the following
{\it spectral data} (\cite{garland-murray}):
\begin{itemize}
\item The two spectral curves, $S_0$ and $S_1$;
\item The ideal ${\cal I}_{S_0\cap S_1}$, decomposing as
${\cal I}_{S_{01}}\cdot {\cal I}_{S_{10}}$; the real structure
interchanges the two factors;
\item An isomorphism of line bundles ${\cal O}[-S_{10}]\otimes
  L^{\mu_0-2\mu_1}\simeq {\cal O}[-S_{01}]$ over $S_0$;
\item An isomorphism of line bundles ${\cal O}[-S_{01}]\otimes
  L^{2\mu_1}\simeq {\cal O}[-S_{10}]$ over $S_1$;
\item A real structure on the line bundle ${\cal
    O}(2k+j-1)[-S_{01}]\otimes L^{\mu_1}|S_1$, lifting the real involution
  $\tau$ on $T\bbp^1$.
\end{itemize}
The structure of this data can be understood in terms of the Schubert
structure of $G/B_0$. The sheaf of principal parts, by definition,
lives in the complement of $U_0$, where the two flags cease to be
transverse. In the case that concerns us here (remember the
invariance under $W$), there are essentially two codimension one
varieties at infinity to $U_0$ that we consider, whose pull-backs give
the two spectral curves. The decomposition of ${S_0\cap S_1}$ into
two pieces is simply a pull-back of the Schubert structure from
$G/B_0$. For the line bundles, the basic idea is that in codimension
one, the principal parts can be understood using embeddings of
$\bbp^1$ into $G/B_0$ given by   principal $Sl(2,\bbc)$s in $G$
(see for instance \cite[Sec. 4]{BHMM}),
reducing the question at least locally to understanding principal
parts for maps into $\bbp^1$, a  classical subject. For simple
poles, the principal part of a map into $\bbp^1$ is encoded by its
residue. Here, the poles are over the spectral curves, and
globally, the principal part is then encoded as a
section of a line bundle over each curve; we identify these below. A
way of seeing that the
spectral data splits into components localised over each curve   is
that we can choose appropriately two parabolic subgroups $Par_0$,
$Par_1$ and project from $G/B_0$ to $G/Par_0$ and $G/Par_1$.

To see what the ``residues'' correspond to here, we exploit a description of
the bundle first given in  \cite[Prop 1.12]{hurtubiseMurray},
the short exact sequence
\begin{equation}\label{description-of-F}
\begin{diagram}
0&\rTo&F&\rTo&
\begin{matrix}\vdots\\ F/(F^0_{p-1,1}+F^\infty_{-p+1,0}) \\ \oplus\\
F/(F^0_{p,0}+F^\infty_{-p,1}) \\ \oplus\\ F/(F^0_{p,1}+F^\infty_{-p,0}) \\ \vdots
\end{matrix}
&\rTo &
\begin{matrix}
\vdots\\ \oplus\\ F/(F^0_{p,0}+F^\infty_{-(p-1),0})\\
\oplus\\ F/(F^0_{p,1}+F^\infty_{-p,1})\\ \oplus\\ \vdots
\end{matrix}
&\rTo &0.
\end{diagram}
\end{equation}
The quotients at the right  are supported over the spectral curves
($F/(F^0_{p,0}+F^\infty_{-(p-1),0})$ over $S_0$, and
$F/(F^0_{p,1}+F^\infty_{-p,1})$ over $S_1$)
and those in the middle are generically line
bundles.

The shift operator $W$ moves the quotients in the middle and
last columns two steps down. The quotients are in fact direct image
sheaves $R^1\tilde\pi_*$ of sheaves on $\tilde{\cal T}$ derived from
$E$; for example, let $E_{p,0, -(p-1),\infty}$ be the sheaf of
sections of $E$ over ${\cal T}$ with poles of order $p$ at ${\cal
  T}_0$ and a zero of order $p-1$ at $\mathcal{T}_\infty$. The sheaf
$F/(F^0_{p,0}+F^\infty_{-p,1})$ can be identified with
$R^1\tilde\pi_*(E_{p,0, -(p-1),\infty})$. One can make similar
identifications for the other sheaves. In a way that
parallels  \cite{hitchinMonopoles}, we have the following result.

\begin{proposition}\label{vanishing-theorem}
a) The spectral curves $S_0$ and $S_1$ are real ($\tau$-invariant);
the quotient $F/(F^0_{p,1}+F^\infty_{-p,1})$ over $S_1$ inherits from
$E$ a quaternionic structure, lifting $\tau$.

b) The sheaves $F/(F^0_{p,0}+F^\infty_{-(p-1),0})$ and
$F/(F^0_{p,1}+F^\infty_{-p,1})$  satisfy a vanishing theorem:
\begin{align*}
H^0\Bigl(T\bbp^1, F/(F^0_{p,0}+F^\infty_{-(p-1),0}) \otimes L^{-z}(-2)\Bigr) =
0&, \text{ for } z\in [(p-1)\mu_0+\mu_1, p\mu_0-\mu_1],\\
H^0\Bigl(T\bbp^1, F/(F^0_{p,1}+F^\infty_{-p,1}) \otimes L^{-z}(-2)\Bigr) = 0&,
\text{ for } z\in (p\mu_0-\mu_1, p\mu_0+\mu_1).
\end{align*}
\end{proposition}
The vanishing hold because these cohomology groups encode $L^2$
solutions to the Laplace equation in the caloron background, which
must be zero; see \cite[Thm 3.7]{hitchinMonopoles} and
\cite[Thm 1.17]{hurtubiseMurray}.

For generic spectral curves, following a line of argument of
\cite{hurtubiseMurray}, Garland and Murray show directly that the
spectral data determines the caloron. To this end, they identify in
\cite[Sec. 6]{garland-murray} the
quotients
\begin{align*}
{F}/{(F^0_{p,0}+F^\infty_{-(p-1),0})} &=
L^{p\mu_0-\mu_1}(2k+j)[-S_{10}]|_{S_0} =
L^{(p-1)\mu_0+\mu_1}(2k+j)[-S_{01}]|_{S_0},\\
F/(F^0_{p,1}+F^\infty_{-p,1}) &=
L^{p\mu_0+\mu_1}(2k+j)[-S_{01}]|_{S_1} =
L^{p\mu_0-\mu_1}(2k+j)[-S_{10}]|_{S_1},\\
F/(F^0_{p,0}+F^\infty_{-p,1}) &=  L^{p\mu_0-\mu_1}(2k+j)\otimes{\cal I}_{S_{10}}, \\
F/(F^0_{p,1}+F^\infty_{-p,0}) &=  L^{p\mu_0+\mu_1}(2k+j)\otimes{\cal I}_{S_{01}}. \\
\end{align*}

These identifications  realised in exact sequence (\ref{description-of-F})
give
\begin{equation}\label{sequence}\begin{diagram}
F&\rInto&\begin{matrix}
\vdots\\ L^{(p-1)\mu_0+\mu_1}(2k+j)\otimes{\cal I}_{S_{01}}\\ \oplus\\
L^{p\mu_0-\mu_1}(2k+j)\otimes{\cal I}_{S_{10}}\\ \oplus\\
L^{p\mu_0+\mu_1}(2k+j)\otimes{\cal I}_{S_{01}}\\ \vdots\end{matrix}
&\rTo&
\begin{matrix}
\vdots\\ L^{(p-1)\mu_0+\mu_1}(2k+j)[-S_{01}]|_{S_0}\\
=L^{p\mu_0-\mu_1}(2k+j)[-S_{10}]|_{S_0}\\ \oplus\\
L^{p\mu_0+\mu_1}(2k+j)[-S_{01}]|_{S_1}\\ = L^{p\mu_0-\mu_1}(2k+j)[-S_{10}]|_{S_1}\\
\vdots\end{matrix}&\rTo&0.
\end{diagram}\end{equation}

The ``residues'' are then the sheaves on the right, supported over the
spectral curves; the maps from the middle to the right hand sheaves
give the various isomorphisms. This diagram shows that
 $F$, and hence the caloron, is encoded in the spectral data.

More generally,  we can define the {\it spectral data}
composed of the curves $S_0, S_1$, the pull-backs to $T\bbp^1$ of the
Schubert varieties in $G/B_0$, and the sheaves on the right hand side
of Sequence (\ref{description-of-F}), with isomorphisms similar to the ones for
generic spectral data given by maps from the sheaves in the middle
column.

To summarise, we have that the caloron determines principal part data, a
section of a sheaf of principal parts supported over the spectral
curve; this data determines the caloron in turn. In the generic case,
the principal part data is equivalent to spectral data can be
described in terms of two curves, their
intersections, and sections of line bundles over these curves. We note
that these generic calorons exist; indeed, we already know that a
caloron is determined by a solution to Nahm's equations, and it is
easy to see that generic solutions to Nahm's equations exist, as we
shall see in Section \ref{moduli}.


\section{From Nahm's equations to spectral data, and back}
\label{sec:NahmSpectralData}
\subsection{Flows of sheaves.}
The solutions to Nahm's equations we consider also
determine equivalent spectral data, as we shall see.
To begin, note that by setting $A(\zeta,z)$ as in Equation
(\ref{eqn:A}), and
\[A_+(\zeta,z) = -iT_3(z) + (T_1-iT_2)(z)\zeta,\]
Nahm's equations (\ref{eqn:Nahm}) are equivalent to the Lax form
\begin{equation}\label{eqn:Lax}
\frac{dA}{dz} + [A_+,A] = 0.
\end{equation}
The evolution of $A$ is by conjugation, so the spectral curve given by
\begin{equation}
det (A(\zeta,z)-\eta\bbi)=0
\end{equation}
in $T\bbp^1$ is an obvious
invariant of the flow we have; if the
matrices are $k\times k$, the curve is a $k$-fold branched cover of
$\bbp^1$. We define for each $z$ a sheaf ${\cal L}_z$
over the curve via the exact sequence
\begin{equation}\label{basic-sequence}
\begin{diagram}
0\rightarrow \OO(-2)^k &\rTo^{A(\zeta,z)-\eta\bbi}& {\cal
  O}^k\rightarrow {\cal L}_z\rightarrow 0.\end{diagram}
\end{equation}
This correspondence taking a solution to Nahm's equations to a
curve and a flow of line bundles over the curve is fundamental to the
theory of Lax equations; see
\cite{Adams-Harnad-Hurtubise}. In our case, we have an equivalence:
\begin{proposition}\label{Nahm-on-interval}
There is an equivalence between

A) Solutions to Nahm's equations on an interval $(a,b)$, given by
$k\times k$ matrices with reality condition  built from skew-hermitian matrices $T_i$
as in Equation (\ref{eqn:A}).

B) Spectral curves $S$ that are compact and lie in the linear system $|\OO(2k)|$
lying in $T\bbp^1$, and flows ${\cal L}_z$, for $z\in (a,b)$, of sheaves
supported on $S$, such that
\begin{itemize}
\item $H^0(T\bbp^1, {\cal L}_z(-1)) = H^1(T\bbp^1, {\cal L}_z(-1)) = 0;$
\item ${\cal L}_z = {\cal L}_{z'}\otimes L^{z-z'}$, for $z,z'\in (a,b);$
\item the curve $S$ is real, that is, invariant under the
  antiholomorphic involution $\tau(\eta,\zeta) =
  (-\bar\eta/\bar\zeta^2, -1/\bar\zeta)$ corresponding to reversal of
  lines in $T\bbp^1$;
\item  there is
  a linear form $\mu$ on $H^0(S,{\cal L}_z\otimes\tau^*({\cal L}_z))$
  vanishing on $H^0(S,{\cal L}_z\otimes\tau^*({\cal L}_z)(-C))$ for all
  fibres $C$ of $T\bbp^1\rightarrow \bbp^1$ and inducing a
  positive definite hermitian metric
$(\sigma_1, \sigma_2)\mapsto \mu(\sigma_1\tau^*(\sigma_2))$ on $H^0(S\cap C,{\cal L}_z)$.
\end{itemize}
\end{proposition}

We start with the passage from A) to B).
For $H^0(T\bbp^1, {\cal L}_z(-1)) = H^1(T\bbp^1, {\cal L}_z(-1)) = 0$ to hold,
we need
$\OO(-3)^k\rTo^{A(\zeta,z)-\eta\bbi}\OO(-1)^k$ to induce an isomorphism
on $H^1$. As shown in \cite[Lemma 1.2]{hurtubiseMurray}, the groups $H^1(T\PP^1,\OO(p))$
are infinite dimensional spaces nicely filtered by
finite dimensional pieces, corresponding to powers of $\eta$ in the
cocycle.
A basis for $H^1(T\bbp^1, {\cal O}(-3))$ is
$1/\zeta,1/{\zeta^2},
{\eta}/{\zeta}, \ldots, {\eta}/{\zeta^4},{\eta^2}/{\zeta},
\ldots,{\eta^2}/{\zeta^6},\ldots$,
and a basis for $H^1(T\bbp^1, {\cal O}(-1) )$ is obtained from it
by multiplying by $\eta$.
Multiplication by $\eta$ then induces an
isomorphism, and so multiplication by $A(\zeta,z)-\eta\bbi$ gives an
isomorphism $H^1(T\bbp^1, \OO(-3)^k)\rightarrow H^1(T\bbp^1, \OO(-1)^k )$.
For the relation ${\cal L}_z = {\cal L}_{z'}\otimes L^{z-z'}$, see
\cite[Sec. 2.6]{hurtubiseMurray} or \cite{Adams-Harnad-Hurtubise} for the general
 theory of the Lax flows. The reality of
the curve follows from the fact that the $T_i$ are skew hermitian. For
the final property, the positive definite inner product on ${\cal
  O}^k$ with respect to which the $T_i$ are skew hermitian induces one
on $H^0(T\bbp^1, {\cal L}_z)$, by passing to global sections in Sequence
(\ref{basic-sequence}). The inner product is a linear map $H^0(S,{\cal
  L}_z)\otimes H^0(S,\tau^*({\cal L}_z))\to \C$. As ${\cal L}_z$ represents,
in essence, the dual to the eigenspaces of the $A(\zeta)^T$, this linear map
factors through $H^0(S,{\cal L}_z \otimes
\tau^*({\cal L}_z))$, and it must be zero on sections that
vanish on fibres.

Now that the passage from A) to B) is established, note that
since $\tau^*L^{z-z'} = L^{z'-z}$, the product ${\cal L}_z\otimes\tau^*({\cal L}_z)$
is constant.  The last condition of B) imposes
severe constraints on ${\cal L}_z$: when
$S$ is smooth, and ${\cal L}_z$ a line bundle, it tells us that
\begin{equation}{\cal L}_z\otimes\tau^*({\cal L}_z)\simeq K_S(2C).
\end{equation}
Indeed, in that case
the vanishing
of the  cohomology of ${\cal L}_z(-1)$ tells us  $\deg({\cal L}_z)=g+k-1$
and $\deg({\cal L}_z \otimes \tau^*({\cal L}_z))=2g+2k-2$. The space of sections
$H^0(S,{\cal L}_z \otimes \tau^*({\cal L}_z))$ is of dimension $g+2k-1$, while
$H^0(S,{\cal L}_z \otimes \tau^*({\cal L}_z)(-C))$ and $H^0(S,{\cal L}_z
\otimes \tau^*({\cal L}_z)(-C'))$ are of dimension $g+k-1$, and
$\dim H^0(S,{\cal  L}_z \otimes \tau^*({\cal L}_z)(-C-C'))=g-1$,
unless it is the canonical bundle, in which case the dimension
is $g$. There cannot be a non-zero form on $H^0(S,{\cal
  L}_z \otimes \tau^*({\cal L}_z))$ vanishing on
$H^0(S,{\cal L}_z \otimes \tau^*({\cal L}_z)(-C))+ H^0(S,{\cal L}_z
\otimes \tau^*({\cal L}_z)(-C'))$ unless  the
intersection $H^0(S,{\cal L}_z \otimes \tau^*({\cal L}_z)(-C-C'))$ is
of dimension $g$, in which case ${\cal L}_z \otimes \tau^*({\cal
  L}_z)(-C-C')$ is the canonical bundle.

Now we deal with the passage B) to A). In essence, it has been
treated by Hitchin in \cite{hitchinMonopoles}, but we give a more invariant
construction of $A$ suitable for our purposes. The fibre product $FP$
of $T\bbp^1$ with itself over $\bbp^1$ is the vector bundle ${\cal
  O}(2)\oplus {\cal O}(2)$, with fibre coordinates $\eta, \eta'$, and
projections $\pi, \pi'$ to $T\bbp^1$; the lift  of
${\cal O}(2)$ to this vector bundle has global sections
$1, \zeta, \zeta^2, \eta, \eta'$, and the diagonal $\Delta$ is cut out
by $\eta-\eta'$. Hence over $FP$ we have the exact sequence
\begin{equation}\label{fibre-sequence}\begin{diagram}
0\rightarrow \OO(-2)&\rTo^{\eta-\eta'}&\OO\rightarrow \OO_\Delta\rightarrow 0.
\end{diagram}\end{equation}
Lifting ${\cal L}_z$ via $\pi$, tensoring it with this sequence, and
pushing it down via $\pi'$ gives
\begin{equation}\label{pushdown-sequence}\begin{diagram}
0\rightarrow V(-2)&\rTo^{\underline{A}(\zeta)-\eta'}&V\rightarrow {\cal L}_z\rightarrow 0
\end{diagram}\end{equation}
for some rank $k$ vector bundle $V$, which over $(\zeta, \eta)
=(\zeta_0,\eta_0)$ is the space of sections of ${\cal L}$ in the fibre
of $T\bbp^1$ over $\zeta=\zeta_0$, and so is independent of
$\eta$: $V$ is lifted from $\bbp^1$. We then identify $V$. The fibre
product $FP = {\cal  O}(2)\oplus {\cal O}(2)$ lives in the product
$T\bbp^1\times
T\bbp^1$, and is cut out by $\zeta-\zeta'$, and one has the exact
sequence
\begin{equation}\label{product-sequence}\begin{diagram}
0\rightarrow \OO(-1,-1)&\rTo^{\zeta-\zeta'}&\OO\rightarrow \OO_{FP}\rightarrow 0.
\end{diagram}\end{equation}
Lifting ${\cal L}_z$ to $T\bbp^1\times T\bbp^1$,
tensoring it with this sequence and taking direct image on the other
factor gives the long exact sequence
\[0\rightarrow H^0(T\bbp^1, {\cal L}_z(-1))\otimes {\cal O}(-1)
\rightarrow  H^0(T\bbp^1, {\cal L}_z)\otimes {\cal O}\rightarrow V
\rightarrow H^1(T\bbp^1, {\cal L}_z(-1))\otimes {\cal O}(-1)\cdots.\]

The vanishing of the cohomolology of ${\cal L}_z(-1)$ shows that
$V\simeq{\cal O}^k$.

We then have given $\underline{A}(\zeta,z)$ as a map of bundles; the process of turning it
into a matrix valued function $A(\zeta,z)$, and of showing that it evolves
according to Nahm's equations as one tensors it by $L^t$,  is given in
\cite[Prop. 4.16]{hitchinMonopoles}. Similarly, the definition of an
inner product on
$H^0(T\bbp^1, {\cal L}_z)$ and the proof that the $T_i$ are skew
adjoint with respect to the inner product follow the line given in
\cite[Sec. 6]{hitchinMonopoles}.

\subsection{Boundary conditions.}\label{sec:boundary-conditions}
Now  suppose that we have a solution to Nahm's equations
satisfying the boundary conditions for an $SU(2)$ caloron, as given
above. We shall show that the boundary behaviour allows us to
identify the initial conditions for the flow, in other words to say
what the sheaves ${\cal L}_z$ are. As  there are two intervals, there
are two spectral curves, $S_0$, which is a $k$-fold cover of $\bbp^1$,
and $S_1$, which is $k+j$-fold. We  suppose the curves are
generic, in that they have no common components and no multiple
components. A solution to Nahm's equations is called generic
if its curves are.
\begin{proposition}\label{nahm-boundary}
There is an equivalence between

A) Generic solutions $A(\zeta, z)$  to Nahm's equations over the
circle, 
denoted $A^0(\zeta, z)$ on
$(\mu_1,-\mu_1+\mu_0)$ and $A^1(\zeta, z)$ on $(-\mu_1,\mu_1)$,
satisfying the  conditions of Section \ref{sec:gaugefields}-B.

B) Generic spectral curves $S_0$ in $T\bbp^1$ of degree $k$ over
$\bbp^1$ that are the support of sheaves ${\cal L}^0_z$,
$z\in [\mu_1,\mu_0-\mu_1]$,
and spectral curves $S_1$ in $T\bbp^1$ of degree $k+j$
over $\bbp^1$ that are the support of sheaves ${\cal L}^1_z$,
$z\in [-\mu_1, \mu_1]$,
with the following properties:
\begin{itemize}
\item
${\cal L}^0_z$, for
$z\in [\mu_1, \mu_0-\mu_1]$, 
${\cal L}^1_z$, for
$z\in(-\mu_1, \mu_1)$,
$S_0$, and $S_1$ have the properties of Proposition \ref{Nahm-on-interval}.
\item
The intersection $S_0\cap S_1$ decomposes as a sum of two divisors
$S_{01}$ and $S_{10}$, interchanged by $\tau$.
\item
At $-\mu_1 =\mu_0-\mu_1$ (on the circle),
${\cal L}^0_{-\mu_1} = {\cal O}(2k+j-1)[-S_{10}]|_{S_0}$
and ${\cal L}^1_{-\mu_1} = {\OO}(2k+j-1)[-S_{10}]|_{S_1}$.
\item
At $\mu_1$, ${\cal L}^0_{\mu_1} = {\cal O}(2k+j-1)[-S_{01}]|_{S_0}$ and
${\cal L}^1_{\mu_1} = {\cal O}(2k+j-1)[-S_{01}]|_{S_1}$.
\item
There is a real structure on ${\cal L}^1_{0}$ lifting $\tau$.
\end{itemize}
\end{proposition}

As a consequence there are isomorphisms of line bundles
$\OO[-S_{01}]\otimes L^{\mu_0-2\mu_1}\simeq \OO[-S_{10}]$
over $S_0$ and
$\OO[-S_{10}]\otimes L^{2\mu_1}\simeq\OO[-S_{01}]$ over $S_1$.

Let us begin with a discussion of the case $k=0$,  studied
by Hitchin in \cite{hitchinMonopoles}. He showed that the solution to Nahm's
equations for an $SU(2)$ monopole corresponded to the flow of line
bundles $L^{t+\mu_1}(j-1)$, $t\in [-\mu_1,\mu_1]$ on the monopole's
spectral curve, with
the flow being regular in the middle of the interval, and having
simple poles at the ends of the interval, with residues giving an
irreducible representation of $SU(2)$.

In the construction given  in Proposition \ref{Nahm-on-interval},
we obtain for this flow a bundle $V$ over $\bbp^1\times [-\mu_1,\mu_1]$, and a
regular section of $Hom(V(-2), V)$ over this interval. For $t\in(-\mu_1,\mu_1)$
the bundle $V$ is trivial on $\bbp^1\times
\{t\}$. The singularity at the end of the interval is caused by a jump
in the holomorphic structure in $V$ at the ends. (Both ends are
identical, because of the identification $\OO = L^{2\mu_1}$ over
the spectral curve.)

To understand the structure of $V$ at ${\cal L}_{-\mu_1}= {\cal O}(j-1)$,
we note that in the natural trivialisations lifted from $\bbp^1$,
local sections of ${\cal O}(j-1)$ over $T\bbp_1$ are filtered by the
order  of vanishing along the zero section: 
\[\OO(j-1)\supset \eta\OO(j-3) \supset \eta^2\OO(j-5)\supset\cdots;\]
this filtration can be turned onto a sum, as we have  the subsheaves
$L_i$ of sections $\eta^i s$, with $s$ lifted from $\bbp^1$. The
construction, by limiting to a curve which is a $j$-sheeted cover of $\bbp^1$,
essentially says that the degrees of vanishing of interest
are at most $j-1$, as one is taking the remainder by division
by the equation of the curve. The bundle $V$ over $\PP^1\times\{\mu_1\}$
decomposes as a sum
\[V ={\cal O}(j-1)\oplus {\cal
  O}(j-3) \oplus {\cal  O}(j-5)\oplus\dots\oplus{\cal O}(-j+1).\]
More generally, for later use, set
\begin{equation*}
 V_{k,k-2\ell}{\buildrel{\rm def}\over{=}}{\cal O}(k)\oplus {\cal
   O}(k-2) \oplus {\cal  O}(k-4)\oplus\dots\oplus{\cal O}(k-2\ell),
\end{equation*}
so $V= V_{j-1,-j+1}$.
Writing $0= \eta^j +\eta^{j-1}p_1(\zeta)+\dots + p_j(\zeta)$
the equation of the spectral curve, with  polynomials $p_i$ of
degree $2i$, we can write the induced automorphism $A(\zeta)-\eta\bbi$
in this decomposition as
\begin{equation*}
\begin{pmatrix}
-\eta& 0 & 0&\dots&0&-p_j(\zeta)\\
1&-\eta& 0&\dots&0&-p_{j-1}(\zeta)\\
0&1&-\eta&\dots&0&-p_{j-2}(\zeta)\\
\vdots&\vdots&\vdots&&\vdots&\vdots\\
0&0&0&\dots&1&-\eta-p_1(\zeta)
\end{pmatrix}
\end{equation*}

There is another way of understanding how the pole of Nahm's
equations generates a non-trivial $V$.
We first restrict to $\zeta = 0$, and vary $z$.
For a moment, suppose by translation that $z=0$ is
the point where the structure jumps. From the boundary conditions,
 the residues at 0 are given in a suitable basis by
\begin{align}\label{residues1}
Res(A_+)&= diag({\frac{-(j-1)}{2}},{\frac{2-(j-1)}{2}},\dots,{\frac{(j-1)}{2}}),\\
Res(A)&= \begin{pmatrix}0&0&\ddots&0&0\\
                        1&0&\ddots&0&0\\
                        0&1&\ddots&0&0\\
                        \vdots&\ddots&\ddots&\ddots&\vdots\\
                        0&0&0&1&0\end{pmatrix}.\label{residues2}
\end{align}
As in \cite[Prop 1.15]{hurtubiseClassification}, we solve
$\frac{ds}{dz} +
A_+s=0$ along $\zeta = 0$, for $s$ of the form
$z^{\frac{j-1}{2}}((1,0,\ldots,0) + z\cdot{\rm holomorphic})$. The
sections $A^j s$ also solve the equation, and, using these sections as a
basis, one conjugates  $A-\eta\bbi$ to
\begin{equation}
\begin{pmatrix}
-\eta& 0 & 0&\dots&0&-p_j(0)\\
1&-\eta& 0&\dots&0&-p_{j-1}(0)\\
0&1&-\eta&\dots&0&-p_{j-2}(0)\\
\vdots&\vdots&\ddots&&\ddots&\vdots\\
0&0&0&\dots&1&-\eta-p_1(0)
\end{pmatrix}
\end{equation}
using a matrix of the form $M(z)=({\rm Holomorphic\ in\ }z) N(z)$, with
\begin{equation}
N(z):=
diag(z^{\frac{-(j-1)}{2}},z^{\frac{2-(j-1)}{2}},\ldots,z^{\frac{(j-1)}{2}}).
\end{equation}
This process can be applied over any point $\zeta$, as the matrices $A(\zeta)$,
$A_+(\zeta)$ have residues (in $z$) conjugate to (\ref{residues1}),
(\ref{residues2}). Indeed, at $\zeta = 0$, they represent standard
generators of a representation of $Sl(2)$, and moving away from $\zeta = 0$
simply amounts to a change of basis. Explicitly, conjugating the residues
of $A(\zeta)$, $A_+(\zeta)$ by $N(\zeta)$ takes them to
the residues of $\zeta A(1)$, $A_+(1)$.
Conjugating again by  a matrix $T$ takes them to the residues of $\zeta A(0), A_+(0)$,
and then by
$N(\zeta)^{-1}$ to the residues of $A(0), A_+(0)$. Thus, for a suitable
\begin{equation}
M( \zeta,z)=({\rm Holomorphic\ in\ }z, \zeta)\cdot N(\zeta^{-1} z)TN(\zeta)
\end{equation}
one has
\begin{equation}
M(\zeta, z)(A(\zeta, z)-\eta\bbi)M(\zeta, z)^{-1} = \begin{pmatrix}
-\eta& 0 & 0&\dots&0&-p_j(\zeta)\\
1&-\eta& 0&\dots&0&-p_{j-1}(\zeta)\\
0&1&-\eta&\dots&0&-p_{j-2}(\zeta)\\
\vdots&\vdots&\vdots&&\vdots&\vdots\\
0&0&0&\dots&1&-\eta-p_1(\zeta)
\end{pmatrix}.
\end{equation}
At $z= 0$, it is a holomorphic section of $Hom(V(-2),V)$ in standard trivialisations.

The same procedures work
when $k\ne 0$.  Indeed, the singular solution to
$\frac{d\\ s}{dz}+A_+s$ has the same behaviour. Starting
from a solution to Nahm's equations, and integrating the connection
as in \cite{hurtubiseClassification}, we find that, near $\mu_1$ on the interval
$(-\mu_1,\mu_1)$, a change of basis of the form
\begin{equation}
M( \zeta,z)=(\text{holomorphic in }z,\zeta)\cdot diag (\bbi_{k\times k},
N(\zeta^{-1} z)TN(\zeta))
\end{equation}
conjugates $A^1(\zeta,z)-\eta\bbi$ to the constant (in $z$) matrix
\begin{equation}\label{normal-form}
 \begin{pmatrix}
a_{11}(\zeta)-\eta& a_{12}(\zeta)&\dots& a_{1k}(\zeta)&0&0&0&\dots&0&g_1(\zeta)\\
a_{21}(\zeta)& a_{22}(\zeta)-\eta&\dots& a_{1k}(\zeta)&0&0&0&\dots&0&g_2(\zeta)\\
\vdots&\vdots&&\vdots&\vdots&\vdots&\vdots&&\vdots&\vdots\\
a_{k1}(\zeta)& a_{k2}(\zeta)&\dots& a_{kk}(\zeta)-\eta&0&0&0&\dots&0&g_k(\zeta)\\
f_1(\zeta)&f_2(\zeta)&\dots &f_k(\zeta)&-\eta& 0 & 0&\dots&0&-p_j(\zeta)\\
0&0&\dots& 0&1&-\eta& 0&\dots&0&-p_{j-1}(\zeta)\\
0&0&\dots& 0&0&1&-\eta&\dots&0&-p_{j-2}(\zeta)\\
\vdots&\vdots&&\vdots&\vdots&\vdots&\vdots&&\vdots&\vdots\\
0&0&\dots& 0&0&0&0&\dots&1&-\eta-p_1(\zeta)
\end{pmatrix}.
\end{equation}
Setting $\mathcal{C}(p(\zeta))$ to denote the companion matrix of $p(\zeta)$, we
write this matrix schematically as
\begin{equation}\label{normal-form-schem}
A^1(\zeta)-\eta\bbi =\begin{pmatrix}
A^0(\zeta)-\eta\bbi&\begin{pmatrix}0& G(\zeta)\end{pmatrix}\\
\begin{pmatrix}F(\zeta)\\ 0\end{pmatrix}&
\mathcal{C}(p(\zeta))-\eta\bbi\end{pmatrix}.
\end{equation}

Let $M_{adj}$ denote the classical adjoint of $M$, so that $M_{adj}M
= \det(M)\bbi$. Then
\begin{equation}\label{determinant}\det(A^1(\zeta)-\eta\bbi) =
  \det(A^0(\zeta)-\eta\bbi) \Bigl(\eta^j+ \sum \eta^{j-i}p_i(\zeta)\Bigr)+ (-1)^j
  F\bigl(A^0(\zeta)-\eta\bbi\bigr)_{adj}G.\end{equation}

The matrix $A^0(\zeta) $ is equal to the limit $A^0(\zeta,\mu_1)$.
 At the boundary point $\mu_1$, the bundle $V^0_{\mu_1}$
 for $S_0$ is trivial, since the solution on $S_0$ is smooth at that
 point; one has
\begin{equation}\label{fibre-sequence2}\begin{diagram}
0\rightarrow {\cal
  O}(-2)^{k}&\rTo^{A^0(\zeta)-\eta\bbi}&
{\cal O}^k\rightarrow {\cal L}^0_{\mu_1}\rightarrow 0.
\end{diagram}\end{equation}

The limit bundle   $V^1_{\mu_1}$ for $S_1$ is
\begin{equation}
V^1_{\mu_1} = {\cal O}^k\oplus V_{j-1,-j+1}
\end{equation} with
\begin{equation}\begin{diagram}
0\rightarrow V^1_{\mu_1}(-2)&\rTo^{A^1(\zeta)-\eta\bbi}&
V^1_{\mu_1}\rightarrow {\cal L}^1_{\mu_1}\rightarrow 0.
\end{diagram}\end{equation}

We now want to identify the limit bundles ${\cal L}^0_{\mu_1}$, ${\cal
  L}^1_{\mu_1}$, and the gluing between them. 
There is a divisor $D$ contained in the intersection $S_0\cap S_1$ such that
${\cal L}^0_{\mu_1}\simeq {\cal O}(2k+j-1)[-D]|_{S_0}$ and  ${\cal
  L}^1_{\mu_1}\simeq {\cal O}(2k+j-1)[-D]|_{S_1}$, and the
correspondence between the matrices is mediated by the maps
 \[{\cal O}(2k+j-1)[-D]|_{S_0}\leftarrow {\cal O}(2k+j-1)\otimes {\cal
   I}_D\rightarrow {\cal O}(2k+j-1)[-D]|_{S_1}.\]
To prove its existence, we
 need to understand how to lift and push down  ${\cal
   O}(2k+j-1)\otimes {\cal I}_D$ (here ${\cal I}_D$ is the sheaf of
 ideals of $D$ on $T\bbp^1$) and so ${\cal O}(2k+j-1)$ through
Sequence (\ref{fibre-sequence}). To do so, we compactify
 $T\bbp^1$ by embedding it into $\bbt = \bbp({\cal O}(2)\oplus{\cal
   O})$, adding a divisor at infinity $P_\infty$. Let $C$ be the
 fibre, then $P_\infty + 2C$ is linearly
 equivalent to the zero section $P_0$ of $T\bbp^1$. Our fibre product
 now compactifies to a $\bbp^1\times\bbp^1$ bundle over
 $\bbp^1$. Similarly, Sequence (\ref{fibre-sequence})
 compactifies to
\begin{equation}\label{fibre-sequence3}\begin{diagram}
0\rightarrow \OO(-2C-P_\infty-P_\infty')&\rTo^{\eta-\eta'}
    &\OO\rightarrow {\cal O}_\Delta\rightarrow 0.
\end{diagram}\end{equation}

The line bundle ${\cal O}(mC+ nP_\infty)$ over
$\bbt$ has over each fibre $C$ a $(n+1)$-dimensional space of
sections, and these sections are graded by the order in $\eta$, as
before. Define an $\ell\times (\ell-1)$ matrix $S(\ell,\eta)$ by
\begin{equation}\label{eqn:defS}
S(\ell,\eta) = \begin{pmatrix}
-\eta& 0 &\cdots &\cdots&0\\
1&-\eta& \ddots&&\vdots\\
0&1&\ddots&\ddots&\vdots\\
\vdots&\ddots&\ddots&-\eta&0\\
\vdots&&\ddots&1&-\eta\\
0&\cdots&\cdots&0&1
\end{pmatrix}.
\end{equation}

\begin{lemma}
Let $m, n>0$. Lifting ${\cal O}(mC+ nP_\infty)$ to the fibre product and
tensoring with Sequence (\ref{fibre-sequence2}),
then pushing down, we obtain an exact sequence
\[\begin{diagram}
0&\rTo &\bigoplus_{i=0}^{n-1}\OO\bigl((m-2i-2)C+(n-i-1)P_\infty\bigr)
&\rTo^{S(n,\eta)}&\bigoplus_{i=0}^{n}\OO\bigl((m-2i)C+(n-i)P_\infty\bigr)\\
&&\phantom{\bigoplus_{i=0}^{n-1}\OO\bigl((m-2i-2)C+(n-i-1)P_\infty\bigr)}&\rTo^{(1,\eta,\eta^2,\ldots,\eta^{n})}& \OO(mC+ nP_\infty)\rightarrow 0.
\end{diagram}\]
Over $T\bbp^1$, it becomes
\begin{equation*}\begin{diagram}
0\rightarrow V_{m-2, m-2n}&\rTo^{S(n,\eta)}
&V_{m, m-2n}&\rTo^{(1,\eta,\eta^2,\ldots,\eta^{n})}&\OO(m)\rightarrow 0.\end{diagram}
\end{equation*}
\end{lemma}
The proof is straightforward. We now want to define a subsheaf of ${\cal O}(2k+j)$. Set
\begin{equation*}R(\zeta,\eta) = \begin{pmatrix}
A^0(\zeta)-\eta\bbi&0\\
\begin{pmatrix}F\\ 0\end{pmatrix}&
S(j,\eta)\end{pmatrix}
\end{equation*}
We define a vector of polynomial functions in $(\zeta, \eta)$
\[(\phi_1,\ldots,\phi_k)= -(-1)^kF(\zeta) (A^0(\zeta)-\eta\bbi)_{adj}.\]
We have
\begin{equation}\label{phi-equation}
(\phi_1,\ldots,\phi_k)\bigl(A^0(\zeta)-\eta\bbi\bigr) +F(\zeta)(-1)^k
\det\bigl(A^0(\zeta)-\eta\bbi\bigr) = 0.
\end{equation}
Write $\phi_i=\sum_{j=0}^{k-1}\phi_{ji}(\zeta)\eta^j$, and
$(-1)^k\det(A^0(\zeta)-\eta\bbi) = \eta^k +
\sum_{j=0}^{k-1}h_j(\zeta)\eta^j$. Decomposing Equation (\ref{phi-equation})
into different powers of $\eta$, we obtain
\begin{equation}\label{phi-equation-matrix}
-\begin{pmatrix}
0&\cdots&0\\
\phi_{01}&\cdots&\phi_{0k}\\
\vdots&&\vdots\\
\phi_{k-1,1}&\cdots&\phi_{k-1,k}
\end{pmatrix} +\begin{pmatrix}
\phi_{01}&\cdots&\phi_{0k}\\
\vdots&&\vdots\\
\phi_{k-1,1}&\cdots&\phi_{k-1,k}\\
0&\cdots&0
\end{pmatrix}  A^0(\zeta) +
\begin{pmatrix}h_0\\ \vdots \\h_{k-1}\\1\end{pmatrix}F= 0.
\end{equation}
Let $M_n$ be the $(k+n)\times(k+n)$ matrix
\begin{equation}\label{M-matrix}
M_n=
\begin{pmatrix}
\phi_{01}&\dots&\phi_{0k}&h_0&0&0\\
\vdots&&\vdots&\vdots&\ddots&0\\
\phi_{k-1,1}&\dots&\phi_{k-1,k}&h_{k-1}&&h_0\\
0&\cdots&0&1&\ddots&\vdots\\
\vdots&&\vdots&0&\ddots&h_{k-1}\\
0&\cdots&0&0&0&1
\end{pmatrix}.\end{equation}
Using Equation (\ref{phi-equation-matrix}), we have the commuting diagram
\[\begin{diagram}
\OO(-2)^{\oplus k}\oplus V_{j-3, -j+1}&\rTo^{R(\zeta,\eta)}
    &\OO^{\oplus k}\oplus V_{j-1, -j+1} &\rTo&{\cal R} \\
\dTo<{M_{j-1}}&&\dTo<{M_j}&&\dTo\\
V_{2k+j-3, -j+1}&\rTo^{S(k+j,\eta)} &V_{2k+j-1, -j+1} &\rTo^{(1,\eta,\ldots,\eta^{k+j-1})}
&{\cal O}(2k+j-1).
\end{diagram}\]
It defines a rank one sheaf $\cal R$, and embeds it in $\OO(2k+j-1)$.
This embedding fails to be surjective when
$(1,\eta,\ldots,\eta^{k+j-1})M_j = 0$, or equivalently when
\[F (A^0(\zeta)-\eta\bbi)_{adj}=0 \text{ and } \det\bigl(A^0(\zeta)-\eta\bbi\bigr)=0.\]
These conditions can only be met on $S_0$, and, because of Equation (\ref{determinant}),
on $S_1$.
In short, there is a subvariety $D$ of $S_0\cap
S_1$, with ${\cal R} ={\cal O}(2k+j-1)\otimes {\cal I}_D$. There are
natural surjective maps from ${\cal R}$ to ${\cal L}_{\mu_1}^0$,
${\cal L}_{\mu_1}^1$: for the projection $\Pi_{m+n,n}$ on a sum with $m+n$ summands onto
the first $n$ summands and the injection $I_{n,n+m}$ into
the first $n$ summands, we have that the diagram of exact sequences
\[\begin{diagram}
\OO(-2)^{\oplus k}&\rTo^{A^0(\zeta)-\eta\bbi}& \OO^{\oplus k}&\rTo&{\cal L}^0_{\mu_1}\\
\uTo<{\Pi_{k+j-1,k}}&&\uTo<{\Pi_{k+j,k}}&&\uTo\\
{\cal O}(-2)^{\oplus k}\oplus V_{j-3, -j+1}&\rTo^{R(\zeta,\eta)}
 &{\cal O}^{\oplus k}\oplus V_{j-1, -j+1}
&\rTo &{\cal R} \\
\dTo<{I_{k+j-1, k+j}}&&\dTo<I&&\dTo\\
{\cal O}(-2)^{\oplus k}\oplus  V_{j-3, -j-1}&\rTo^{A^1(\zeta)-\eta\bbi}
 &{\cal O}^{\oplus k}\oplus V_{j-1, -j+1}
&\rTo &{\cal L}_{\mu_1}^1.
\end{diagram}\]
commutes. This diagram identifies ${\cal L}_{\mu_1}^0$,
${\cal L}_{\mu_1}^1$ as ${\cal O}(2k+j-1)[-D]$ over their respective
curves.

A similar analysis for $-\mu_1$ shows that ${\cal
  L}_{-\mu_1}^0$, ${\cal L}_{-\mu_1}^1$ are ${\cal O}(2k+j-1)[-D']$,
for some $D'$.
To relate $D$ to $D'$, note that as in \cite[Sec. 6]{hitchinMonopoles},
there is a real structure on ${\cal L}_0$ lifting the
real involution $\tau$.  Hence $T_i(0) = T_i(0)^T$.
On the other hand, for a solution $A(z)$ to Nahm's equations (\ref{eqn:Lax}),
\begin{equation}
\frac{dA(-z)^T}{dz} + [A_+(-z)^T,A(-z)^T] = 0.
\end{equation}
As solutions to the same differential equation with the same
initial condition, all the way around the circle,
\begin{equation}
A(-z) = A^T(z).
\end{equation}

This symmetry tells us that in terms of the matrix $A^1(\zeta),
A^0(\zeta)$ at $\mu_1$, the equation for $D'$ is
\begin{equation}
0 = (A^0(\zeta)-\eta\bbi)_{adj} G.
\end{equation}
Using Equation (\ref{determinant}), we see that along $S_0$ where
$\det\bigl(A^0(\zeta)-\eta\bbi\bigr)=0$, the intersection $S_0\cap S_1$
is cut out by $0=F\bigl(A^0(\zeta)-\eta\bbi\bigr)_{adj}G$.
Generically along $S_0$, the matrix
$(A^0(\zeta)-\eta\bbi)$  has corank one, and so
$(A^0(\zeta)-\eta\bbi)_{adj}$ has rank one. We can thus write it
as the product $U V$ of a column vector and a row vector. The
equation for the intersection is then $0=(F U)(V G)$, the
product of the defining relations for $D$ and $D'$. If the two curves
are smooth, intersecting transversally, then $S_0\cap S_1 = D+D'$.
This property holds independently of whether the matrices arise from calorons
or not. As one has the result for generic intersections, varying continuously
gives $S_0\cap S_1 = D+D'$ even for non-generic curves.

The skew Hermitian property of the $T_i$ implies that $A(\zeta, z) =
-\zeta^2\bar A(-1/\bar\zeta, z)^T$.  This symmetry transforms the
equation $0=F(A^0(-1/\bar\zeta)-\eta\bbi)_{adj}$ into
$0=(A^0(\zeta)-\eta\bbi)_{adj}G$, showing that $\tau(D) = D'$.

We have now obtained our complete generic spectral data from our
generic solution to Nahm's equations. The converse, starting with the
spectral data, is essentially done above, apart from the boundary
behaviour. This last piece is dealt with in \cite[Sec. 2]{hurtubiseMurray};
the conditions on the intersections of the spectral curves used here
are more general, but the proof goes through unchanged.

\section{Closing the circle}
\label{sec:closingcircle}
We have two different types of gauge fields, with a transform relating
them; both have complex data associated to them, which encodes
them; this data, we have seen (at least in the generic case) satisfies
the same conditions, with for example the sheaves
$\bigl(F/(F^0_{p,0}+F^\infty_{-(p-1),0}) \bigr)\otimes L^z(-1)$ and
$\bigl(F/(F^0_{p,1}+F^\infty_{-p,1})\bigr)\otimes L^z(-1)$
obtained from the caloron satisfying the same
conditions as   the ${\cal L}^t$ obtained from the solution to Nahm's equation
for appropriate values of $z$ and $t$. We
now want to check that the complex data associated to the
object is the same as that associated to the object's transform.

\subsection{Starting with a caloron.}
Starting with a caloron, one can define as above, the spectral curves
and sheaves over them.
\begin{itemize}
\item The curve $S_0$ is the locus  where
 $F^0_{p,0}$ and $F^\infty_{-p+1,0}$ have non-zero intersection.
\item The curve $S_1$ is the locus  where $F^0_{p,1}$ and
  $F^\infty_{-p,1}$  have non-zero intersection.
\item The quotient $F/(F^0_{p,0}+F^\infty_{-(p-1),0})$, supported
  over $S_0$, is generically isomorphic to the line bundle
  $L^{p\mu_0-\mu_1}(2k+j)[-S_{10}]|_{S_0} =
  L^{(p-1)\mu_0+\mu_1}(2k+j)[-S_{01}]|_{S_0}$.
\item The quotient $F/(F^0_{p,1}+F^\infty_{-p,1})$, supported  over
  $S_1$, is generically isomorphic to the line bundle
  $L^{p\mu_0+\mu_1}(2k+j)[-S_{01}]|_{S_1} =  L^{p\mu_0-\mu_1}(2k+j)[-S_{10}]|_{S_1}$.
\end{itemize}

We can ``shift'' the caloron by the $U(1)$ monopole with constant
Higgs field $iz$. Let us consider the direction in $\bbr^3$,
corresponding to $\zeta = 0$: the positive
$x_3$ direction. If one does this, for $z\in ((p-1)\mu_0+\mu_1, p\mu_0
-\mu_1)$, $F^0_{p,0}$ represents the sections in the kernel of
$\nabla_3-i\nabla_0+z $ along each cylinder that decay  at infinity in
the positive direction;
similarly, $F^\infty_{-(p-1),0}$  represents the sections along each
cylinder that decay at infinity in the negative
direction. For
$z\in (p\mu_0-\mu_1,p\mu_0+\mu_1)$, $F^0_{p,1}$ represents the
sections  along each cylinder that decay at infinity
in the positive direction; similarly, $F^\infty_{-p,1}$  represents
the sections along each
cylinder that decay at infinity in the negative
direction. With these shifts, the spectral curves represent the lines
for which the solutions that decay at one or other of the ends do not
sum to the whole bundle.
The quotients
$F/(F^0_{p,0}+F^\infty_{-(p-1),0})$ and $F/(F^0_{p,1}+F^\infty_{-p,1})$
represent this failure and therefore the existence of a solution decaying at both ends.
Suppose that $z\in
(p\mu_0-\mu_1,p\mu_0+\mu_1)$; the other case is identical.

Following the Nahm transform heuristic, starting with the caloron, we
solve the Dirac equation for the family of connections shifted by $z$, and
obtain in doing so a bundle over an interval and a solution to Nahm's equations
on this bundle.
The spectral curve, from
this point of view, over $\zeta = 0$, consist of the eigenvalues of
$A(0, z) = T_1+iT_2 (z)$; there is a sheaf over the curve whose
fibre over the point $\eta$ in the spectral curve given by the
cokernel of $A(0, z)-\eta\bbi$.

A priori, it is not evident what link there is between eigenvalues on
a space of solutions to an equation defined over all of space, and the
behaviour of solutions to an equation along a single line. The link is
provided by a remarkable formulation of the Dirac equation; see for instance
\cite{DK} or \cite{cherkis2001}.
The idea, roughly, is to write the Dirac equation as the equations for
the harmonic elements of the complex
\begin{equation}\label{complexSR}\begin{diagram}
L^2(V)&\rTo^{D_1 = \begin{pmatrix} \nabla_3+i\nabla_0+z \\
                                  -\nabla_1-i\nabla_2\end{pmatrix}}&L^2(V)^{\oplus 2}
&\rTo^{D_2 = \begin{pmatrix}\nabla_1+i\nabla_2& \nabla_3+i\nabla_0+z  \end{pmatrix} }
&L^2(V).\end{diagram}\end{equation}

\begin{proposition} \label{prop:Dirac-cohomology}
The operators $D_1, D_2$ commute with multiplication
by $w= x_1+ix_2$, and for self-dual connections, $D_2D_1 =0$.
The kernel  $K_z =\ker(D_2)\cap \ker(D_1^*)$ of the Dirac operator $D^*_z$
is naturally isomorphic to the cohomology $\ker(D_2)/\im(D_1)$ of the
complex.
\end{proposition}

On a compact manifold, this equivalence is part of
standard elliptic theory. In the non-compact case, the analysis must
be done with  care. The
simplest way, for us, is to adapt the work of Nye and Singer
\cite[Sec. 4]{singernye}. Using the techniques developed by Mazzeo
and Melrose \cite{mazzeomelrose}, they show that the Dirac operator
$D^*_z$ is Fredholm if and only an ``operator on fibers along the
boundary'' $P = (\nabla_0-z)_\infty + i(\eta_1e_1+\eta_2e_2+\eta_3e_3)$
is invertible along all the circles $S^1$ in $S^1\times S^2_\infty =
\partial(S^1\times \bbr^3) $. Again,
the $e_1$ are the Pauli matrices, and the constraint must hold
for all choices of constants $(\eta_1,\eta_2, \eta_3)$. Nye and Singer
show that  it does as long as $z$ does not have
values $\mu_i + n\mu_0$. In the the elliptic complex, the
constraint gets replaced by the essentially equivalent condition that
the complex
\[\begin{diagram}
L^2(V|_{S^1})&\rTo^{\begin{pmatrix} \eta_3+i(\nabla_0)_\infty+z \\
-\eta_1-i\eta_2\end{pmatrix}}& L^2(V|_{S^1})^{\oplus 2}&\rTo^{
\begin{pmatrix}\eta_1+i\eta_2& \eta_3+i(\nabla_0)_\infty+z\end{pmatrix} }&L^2(V|_{S^1})
\end{diagram}\]
defined over the circle be exact. The equivalence between the Euler
characteristic of the complex and the index of the Dirac operator then
goes through, establishing the isomorphism.

We now turn to analysing the way a solution to Nahm's equations is extracted
from this complex. The operator $A(0,z)$ on the kernel is defined by
multiplication by $x_1+ix_2$ followed by the projection on that kernel.
On the cohomology, it is simply multiplication by $w=x_1+ix_2$, simplifying
matters considerably.

Suppose $\eta$ is an eigenvalue  of $A(0,z)$, hence there
 is a  cohomology class $v$ such that
$(x_1+ix_2-\eta)(v) = D_1(s)$ for some $s$ in $L^2$.
In particular, along $w=\eta$ in $\SR$,
we have $(\nabla_3+i\nabla_0-z)s = 0$,
and so  $(\eta,\zeta)=(w,0)$ belongs to the spectral curve.
Conversely, if $(\nabla_3+i\nabla_0-z)s = 0$ along
$w=\eta$, we can build a cohomology class $v$ by extending
$s$ to a neighbourhood so that it is compactly supported in the $x_1,
x_2$ directions and satisfies $(\nabla_1+i\nabla_2)s = 0$ along
$w=\eta$; one then sets $v = \frac{D_1s}{w-\eta}$. The spectral curves
are thus identified.

We can further identify sections of the quotients 
$F/(F^0_{p,1}+F^\infty_{-p,1})$
(for the caloron twisted by $z$)
over $S_1\cap \{\zeta = 0\}$ with sections of the
the cokernel of $A^1(0,z)-\eta\bbi$ as both as supported over 
$S_1\cap \{\zeta=0\}$.
Indeed, let $\sigma_\epsilon(w)$ and $\rho(x_3)$ be smooth functions such that
\[\sigma(w)=\begin{cases}1, &|w|<\epsilon;\\ 0, &|w|>\epsilon;\end{cases}
\quad\text{ and }\quad
\rho(x_3)=\begin{cases}0, &x_3<0;\\ 1, &x_3>1.\end{cases}\]
Notice that the derivative $(\nabla_1+i\nabla_2)(\sigma_\epsilon(w))$ 
is supported
on the annulus $\epsilon<|w|<2\epsilon$.

Consider a ball $B$ centred at $\eta$ in the fiber above $\zeta=0$ of radius
$2\epsilon$ chosen such that $B\cap S_1=\{\eta\}$.  The ball parameterises
cylinders $S^1\times \{w\}\times \R$, for $w\in B$, and thus 
a section $\phi\in H^0(B,F)$
is in fact a section of $V$ on $S^1\times B\times \R$ satisfying
$(\nabla_3+i\nabla_0+z)\phi=0$.  The holomorphicity condition is then
$(\nabla_1+i\nabla_2)\phi=0$.

Let $\phi_0\in H^0(B,F_{p,1}^0|_{(0,\eta)})$. Then
$\sigma(w-\eta)\rho(x_3)\phi_0$ lies in $L^2(\SR,V)$.
Thus the section
$D_1(\sigma(w-\eta)\rho(x_3)\phi_0) $ is a coboundary for Complex (\ref{complexSR}).
Similarly, if $\phi_\infty$ lies in $F^\infty_{-p,1}$, then
$\sigma(w-\eta)(\rho(x_3)-1)\phi_\infty$ also lies in $L^2$, and its image by $D_1$
is also a coboundary.

Suppose for simplicity $S_1\cap \{\zeta=0\}$ contains only points of multiplicity
one.
Consider now a general section $\phi\in H^0(B,F)$.  Away from the spectral curve, $\phi$
decomposes into a sum $\phi_0+\phi_\infty$, and combining the two
constructions above gives a coboundary $K(\phi)$ for Complex (\ref{complexSR}).  This
decomposition has a pole at $\eta$ if $\phi(\eta)$ represents a
non trivial element in the quotient $F/(F^0_{p,1}+F^\infty_{-p,1})$.  However,
the section
\begin{align*}
K(\phi)&=D_1\bigl(\sigma_\epsilon(w-\eta)\rho(x_3)\phi-\rho_\infty\bigr)\\
    &=\begin{pmatrix}
         \sigma_\epsilon(w-\eta)\bigl(\nabla_3+i\nabla_0+z\bigr)
         \bigl(\rho(x_3)\phi-\rho_\infty\bigr)\\
      -\bigl(\nabla_1+i\nabla_2\bigr)
            \bigl(\sigma_\epsilon(w-\eta)\rho(x_3)\phi-\rho_\infty\bigr)
\end{pmatrix}\end{align*}
is $L^2$.  Because of the pole of $\phi_0+\phi_\infty$, it is not in $D_1(L^2(V))$ but
by construction it is definitely in the kernel of $D_2$.  It thus
represents a non-trivial cohomology class for Complex (\ref{complexSR}).

Doing this for each point of $S_1\cap \{\zeta =0\}$ expresses $K_z$ as a sum
of classes localised along the lines in $S^1\times \bbr^3$ corresponding to
the intersection of the spectral curve with $\zeta = 0$, identifying $K_z$
with sections of  $F/(F^0_{p,1}+F^\infty_{-p,1})$ over $\zeta = 0$. Let
$K_z(\phi)$ correspond to a non-zero element $\phi$ of
$F/(F^0_{p,1}+F^\infty_{-p,1})$ over $(\eta, 0)$.

This class projects non-trivially to the cokernel of $A(0,z)-\eta\bbi$.
Indeed, the action of $A(0,z)$ on $K_z$ is multiplication by $w=x_1+ix_2$,
so if $K_z(\phi)=(A(0,z)-\eta\bbi)K_z(\sigma)= (w-\eta)K_z(\sigma)$ for some
$\sigma$ with $\phi= (w-\eta) \sigma$, then the decomposition
$\phi=\phi_0+\phi_1$ is holomorphic and so $K_z(\phi) = 0$.

For intersections of higher multiplicity, the case presents
only notational difficulties, but is conceptually identical.
The choice of $\zeta$ made above just simplified the analysis, varying
it we obtain the following result.

\begin{proposition}
For $z\in (p\mu_0-\mu_1,p\mu_0+\mu_1)$, the curves $S_1$ associated
to the caloron and $S_1'$ associated to its Nahm transform
are identical; there is a natural isomorphism between the
 sheaf $F/(F^0_{p,1}+F^\infty_{-p,1})$ for the caloron shifted by $z$
and the cokernel sheaf of $A^1(\zeta, z)-\eta$.

Similarly, for $z\in ((p-1)\mu_0+\mu_1, p\mu_0-\mu_1)$, the curves
$S_0$ associated to the caloron and  $S_0'$ associated to its
Nahm transform are identical; there is a natural isomorphism between
the sheaf $F/(F^0_{p,0}+F^\infty_{-(p-1),0})$ for the caloron shifted by $z$
 and the cokernel sheaf of $A^0(\zeta, z)-\eta$.
\end{proposition}

We remark, in particular, for generic solutions, that flowing to a
boundary point gives the divisor $D$ in the intersection of the two
curves, and so an identification of the spectral data.

\subsection{Starting with a solution to Nahm's equation.}
We now want to identify the spectral data for a solution to Nahm's equations 
and the caloron it produces.

{\bf The case $j>0$.}  Again, as for the opposite direction, it is be
more convenient to give a cohomological interpretation of the Dirac
equation. This interpretation is governed by choosing a particular
direction in $\bbr^3$, and so a particular equivalence of
$\SR$ with $\bbc^*\times \bbc$. We choose the direction
corresponding to $\zeta=0$, so set $\beta = T_1+iT_2$, $w =
x_1+ix_2$, $\alpha =T_0-iT_3$, $y= \mu_0x_0 -ix_3$,
and consider the complex
\begin{equation}\label{Dirac-complex}\begin{diagram}
\tilde L^2_2 &\rTo^{D_1 = \begin{pmatrix} i(d_z+\alpha-iy)\\
      -\beta+iw\end{pmatrix}}&(\tilde L^2_1)^{\oplus 2}&
\rTo^{D_2 = \begin{pmatrix}\beta-iw & i(d_z+\alpha-iy)\end{pmatrix}}
         &\tilde L^2.\end{diagram}\end{equation}
For this complex to be defined, we must display a bit of care in our choice
of function spaces:
\begin{itemize}
\item $\tilde L^2_2$ is the space of  sections $s$ of $K$ that are $L^2_2$ over each
interval and such that, for subscripts denoting the limits
on the appropriates sides of the jump points,
  $\pi(s_{\rm large}) =s_{\rm small}$,  $\pi(d_zs_{\rm large})=d_zs_{\rm small}$;
\item $(\tilde L^2_1)^{\oplus 2}$ is the space of sections $(s_1, s_2)$ of
$K\oplus K$ that are in $L^2_1$ over each interval, with
$\pi(s_{\rm large}) = s_{\rm small}$ at the jump points;
\item $\tilde L^2$ is the space of $L^2$ sections of $K$.
\end{itemize}

\begin{proposition}\label{prop:D1D1}
For solutions to Nahm's equations, $D_2D_1=0$, and
\begin{equation}D_1^*D_1=-\bigl(d_z+(T_0-i\mu_0x_0)\bigr)^2 -
                \sum_{j=1}^3 (T_j-ix_j)^2.\end{equation}
So for $v\ne 0 $ satisfying  $D_1^*D_1(v)=0$, there is a point-wise relation
\begin{equation}\label{eqn:convexity}d_z^2\|v\|^2  =
2\|\bigl(d_z+(T_0-i\mu_0x_0)\bigr)v\|^2+ 2\sum_{j=1}^3\|(T_j-ix_j)v\|^2\ge 0.
\end{equation}
Hence if $v$ is a regular solution vanishing at two points, $v=0$.
\end{proposition}

We look at solutions to the equations $D_1^*D_1(v)=0$ near a boundary
point $\pm \mu_1$; one has  on the interval $(-\mu_1, \mu_1)$ near
$\pm \mu_1$, a decomposition $\bbc^{k+j} = \bbc^k\oplus\bbc^j$; the
$2k$ dimensional space of solutions with boundary values (and
derivatives) in $\bbc^k$ continue outside the interval, and solutions
with boundary values in $\bbc^j$ are governed by the theory of regular
singular points for o.d.e.s. As the sum of the squares of the residues
at the boundary points  of the $T_i$ is half of the Casimir element in the
enveloping algebra of $su(2)$,
its value is $-(j-1)^2/4$. The theory of regular singular o.d.e. then
gives, for $\rho$ one  of $(1\pm j)/2$, $j$ dimensional spaces of
solutions of the form $\hat z^\rho f(\hat z)$ to $D_1^*D_1(v)=0$,
where $f(\hat z)$ is analytic, at each of the boundary points. Here
$\hat z$ is a coordinate whose origin is at the boundary point.

We can use this knowledge to build a Green's function. This operation is somewhat
complicated by the poles of the $T_i$, as we now see in a series of lemmas.

\begin{lemma}           
Let $x\in (-\mu_1, \mu_1)$, and $u\in \bbc^{k+j}$. There is a unique
solution to $D_1^*D_1(v)=\delta_xu$ on the circle. \end{lemma}

\begin{proof}
Such a solution must be continuous, smooth on the circle with a single jump in the
derivative at $x$, of value $u$. Let $U$ be the $2k$-dimensional space of solutions
on the small side, outside $(-\mu_1, \mu_1)$.  Those solutions propagate into
the interval from both ends, and for $f\in U$, let $f_-$ be the continuation into
the interval $(-\mu_1,\mu_1)$ from the $-\mu_1$ side, and $f_+$ from the $\mu_1$ side.
Let $V_+$ and $V_-$ be the $j$-dimensional spaces of solutions
of the form $\hat z^{(1+j)/2} f(\hat z)$ respectively born at $\mu_1$ and  $-\mu_1$.
Consider the map
\begin{align*}
R_x\colon U\oplus V_+\oplus V_-&\rightarrow \bbc^{k+j}\oplus \bbc^{k+j}\\
(f, g_+, g_-)&\mapsto (f_-(x)+g_-(x)-f_+(x)-g_+(x),
f'_-(x)+g'_-(x)-f'_+(x)-g'_+(x)).
\end{align*}
This map must be injective (thus bijective) because of the convexity property of solutions
given by Equation (\ref{eqn:convexity}). If $R_x^{-1}(0,u)=(f,g_+,g_-)$, the
desired Green's function $v$ can be chosen by taking $f$ outside of $(-\mu_1,
\mu_1)$, $f+g_-$ on $(-\mu_1, x)$, and $f+g_+$ on $(x, \mu_1)$.
\end{proof}      

\begin{lemma}           
Let $u_+, u_-\in \bbc^{k}$. There is a unique solution to
$D_1^*D_1(v)=0$ on $(-\mu_1, \mu_1)$ with values $u_+, u_-$ at
$\mu_1,-\mu_1$ respectively. \end{lemma}

\begin{proof}Let $W_-$ and $W_+$ be the $k+j$-dimensional affine spaces of
 solutions  with boundary value $u_-$ at $-\mu_1$, and 
$u_+$ at $\mu_1$ respectively. Consider at an intermediary
 point $x$ the affine map
\begin{align*}
R_x\colon W_+\oplus W_-&\rightarrow \bbc^{k+j}\oplus \bbc^{k+j}\\
(g_+, g_-)&\mapsto (g_-(x)-g_+(x),  g'_-(x)-g'_+(x)).
\end{align*}
It is injective, otherwise one has, taking
differences, an element of $\ker(D_1^*D_1)$
vanishing at both ends; it is thus surjective,
and the inverse image of zero gives solutions whose
values and derivatives match at $x$.
\end{proof}      

A similar result holds on the other interval.
\begin{lemma}           
Let $x\in (\mu_1, \mu_0 -\mu_1)$, $u_+, u_-\in \bbc^{k}$ and $u\in
\bbc^{k}$. There is a unique solution to $D_1^*D_1(v)=\delta_xu$ on
the interval, with value $u_+$ at $\mu_1$, and value $u_-$  at $\mu_0
-\mu_1$. \end{lemma}

Combining the two previous lemmas, we obtain the Green's function for the small side.
\begin{lemma}           
Let $x\in (\mu_1, \mu_0 -\mu_1)$ and $u\in \bbc^{k}$. There is a
unique solution to $D_1^*D_1(v)=\delta_xu$ on the circle. \end{lemma}

\begin{proof}Again, one has, varying the boundary values at
$\pm\mu_1$, an affine $2k$ dimensional space of continuous solutions,
with the right jump in derivatives at $x$, and possibly extra
jumps in the derivatives at $\pm\mu_1$. One considers the affine map
from this space to $\bbc^k\oplus\bbc^k$ taking the jumps in the
derivatives at $\pm\mu_1$; it has to be injective (convexity) and so
is surjective, allowing us to match the derivatives.\end{proof}

By the usual ellipticity argument, the Green's function solution to
$D_1^*D_1(v)=\delta_xu$ is smooth away from $x$.  By convexity, $\|v\|$
and $d_z\|v\|$ are increasing everywhere away from $x$.  Since we are
on a circle, both must attain their maximum at $x$.
In addition, integrating over the circle, we get
\begin{align*}
|u|^2&\ge  d_z<v,v>(x)_- -d_z<v,v>(x)_+ \\
     &= \int d_z^2<v,v> \\
     &= \int 2\|(\nabla_0-i\mu_0x_0)v\|^2+ 2\sum_j \| (T_j-ix_j)v\|^2\\
     &\ge C\int \|(\nabla_0-i\mu_0x_0)v\|^2+\|v\|^2.
 \end{align*}
The last inequality follows from the fact that the solution to Nahm's equation is
irreducible. This $L^2_1$ bound
on $v$ ensures continuity, with $\|v\|_{L^\infty}\leq C \|v\|_{L^2_1}$, and hence
the $L^\infty$ norm of the Green's function is  bounded by a constant
times the norm of $u$.

{\bf The case $j=0$}.  We can again use a cohomological version of
the Dirac operator, but we must modify the function spaces a little bit to
account for the jump at $\mu_1$ in the solution to Nahm's equations given by
$\Delta_+ (A(\zeta)) = (\alpha_{+,0}+\alpha_{+,1}\zeta)
                       (\bar\alpha_{+,1}^T- \bar\alpha_{+,0}^T\zeta)$ for
suitable column vectors $\alpha_{+,i}$, and the similar jump
 $\Delta_- (A(\zeta)) = (\alpha_{-,0}+ \alpha_{-,1}\zeta)(
\bar\alpha_{-,1}^T-  \bar\alpha_{-,0}^T\zeta)$ at $-\mu_1$.
The corresponding jumps for the matrices $T_i$ are
\begin{align*}
\Delta_\pm(T_3) &= \frac{i}{2} (\alpha_{\pm,1}   \bar\alpha_{\pm,1}^T -
\alpha_{\pm,0}   \bar\alpha_{\pm,0}^T), \\
\Delta_\pm(T_1+iT_2) &= \alpha_{\pm,0}   \bar\alpha_{\pm,1}^T.
\end{align*}

We want solutions to the Dirac equations $D_1^*(s_1, s_2)= D_2(s_1, s_2)=0$ with jumps
that are multiples of
$v_\pm := (i\alpha_{\pm,1},\alpha_{\pm,0})$ at $\pm\mu_1$, so  we modify the
 function spaces for  Complex (\ref{Dirac-complex}):
\begin{itemize}
\item $\tilde L^2_2$ is the space of  sections $s$ of $K$ that  are
  $L^2_2$ over each interval, continuous at the jumping points, and
  with a jump discontinuity $\Delta_\pm (d_zs) = \frac{1}{2}
  (\alpha_{\pm,1}   \bar\alpha_{\pm,1}^T +\alpha_{\pm,0}\bar\alpha_{\pm,0}^T)s(\pm\mu_1)$
  in derivatives at $\pm\mu_1$;
\item $(\tilde L^2_1)^{\oplus 2}$ is the space of sections $(s_1, s_2)$ of
  $K\oplus K$ that are $L^2_1$ over the intervals, but with a jump
  discontinuity $\Delta_\pm(s_1, s_2) = c_\pm v_\pm$ at $\pm\mu_1$;
\item $\tilde L^2$ is the space of $L^2$ sections of $K$.
\end{itemize}

\begin{proposition}For solutions to Nahm's equations and for the complex defined using
those function spaces just defined, Proposition \ref{prop:D1D1} holds,
and furthermore, at the jump points,
\begin{equation}\Delta(d_z(<v,v>)) =  <\Delta(d_z v),v> + <v,
  \Delta(d_z v)> =  (<\bar \alpha_{\pm,0}^Tv,\bar \alpha_{\pm,0}^Tv >+
  <\bar \alpha_{\pm,1}^Tv,\bar \alpha_{\pm,1}^Tv >)\ge 0
\end{equation}
for $v\ne 0 $ satisfying  $D_1^*D_1(v)=0$; hence, if   $v$ is a
regular solution vanishing at two points, then $v=0$.
\end{proposition}

Proceeding as for $j>0$, one can then build a Green's function. Using
the Green function in both cases $j>0$ and $j=0$, we have the analog of
Proposition \ref{prop:Dirac-cohomology}.

\begin{proposition}
The cohomology $\ker(D_2)/\Im(D_1)$ of Complex (\ref{Dirac-complex})
is isomorphic to the space of $L^2$ solutions to the Dirac equation
$D_2(v)=0=D_1^*(v)$, which are the harmonic representatives in the
cohomology classes. As such, they have minimal norm in the class.
\end{proposition}

\begin{proof}The cohomology class of a solution to the Dirac equation is non zero
as the convexity property ensures $D_1^*D_1(v)=0$ over the whole circle implies $v=0$.
On the other hand, from a solution to
$D_2(u)=0$, one uses the Green's function to solve $D_1^*D_1(v)=D_1^*(u)$
and gets a harmonic representative $u+D_1(v)$.\end{proof}



We can now turn to identifying the spectral data of
a solution to Nahm's equations with the data of the caloron it
induces, beginning with the spectral curves.
Suppose we are in the generic situation
of spectral curves that are reduced and with no common
components. Let's work over $\zeta = 0$, and
suppose,  that the intersection of $\zeta =
0 $ with the curves is generic, so  distinct points of
multiplicity one.

We turn first to analysing the kernel of the matrices $\beta(z)-iw =
T_1(z)+iT_2(z) - ix_1 +x_2$, at points $w$ where $\det(\beta(z)-iw)=
0$ so that we are on the Nahm spectral curve, let us say for the
interval $(-\mu_1,\mu_1)$. We note that since $[\beta, d_z+\alpha]=0$,
the spectrum is
constant along the intervals $(-\mu_1, \mu_1)$ or
$(\mu_1,\mu_0-\mu_1)$, and indeed, if $(\beta(z_0)-iw )f_0=0$,
solving $(d_z+\alpha)f = 0$, $f(z_0) = f_0$ gives $(\beta(z)-iw) f = 0$
over the interval.

For such an $f$ and a bump function $\rho$, setting
\begin{equation}
(s_1, s_2) = (\rho f, 0)
\end {equation}
defines a cohomology class in Complex (\ref{Dirac-complex}). This class is
zero only when the integral of
$\rho$ is zero over the interval as then for some $\sigma$, we have
$(d_z + \alpha)(\sigma f) = ((d_z\sigma)f, 0) =(\rho f,0)$.
Note that we can use a boundary and do $\rho\mapsto \rho +d_z\sigma$
to place the bump
anywhere on the interval. This fact is quite useful.

Given this class at a fixed $y = 0$, one can
extend it to other $y$ by taking
\begin{equation}\label{extension}
(s_1, s_2) = (e^{iyz}\rho f, 0).
\end {equation}

We can think of this family of cohomology classes as a section
over $\SR$ of the bundle that Nahm's construction produces from the
solution to Nahm's equations.
This section has exponential decay as
$x_3\rightarrow\pm\infty$. To see this, we
exploit the fact that we can move the bump function $\rho$
around to give different representatives
supported  on $(-\mu_1,\mu_1)$ for the class
$(s_1, s_2) = (e^{iyz}\rho f, 0) = (e^{-x_3z}e^{i\mu_0x_0}\rho f)$.
Recall that the norm of any representative in the cohomology
class is at least the norm of the harmonic (Dirac) representative. As we
take $x_3$ to $+\infty$, we move the bump towards $\mu_1$, giving up to a
polynomial the bound $exp (-x_3\mu_1)$ on the $L^2$ norm of the harmonic
representative; similarly, as we go to $-\infty$, we move the bump towards
$-\mu_1$, and get a bound of the form $exp (x_3\mu_1)$.

The condition on a  section of $E$ over $\C^*=\{e^{i\mu_0x_0 + x_3}\}$
to extend over $0, \infty$ is given in
\cite[Sec. 3]{garland-murray}, and amounts to a growth condition. More
generally, there are a whole family of growth conditions, giving us not only
our bundle $E$ but a family of associated sheaves $E_{(p,q,0)
(p',q',\infty)}$,  the sheaf of sections
of $E$ over ${\cal T}$ with poles of order $p$ at ${\cal T}_0$, with
leading term in $E^0_q$,  and poles of order $p'$ at ${\cal T}_\infty$, with
leading term in $E^\infty_{q'}$.
The growth condition that $E_{(0,1,0) (0,1,\infty)}$ corresponds to is that
of  growth bounded (up to polynomial factors) by $exp (-x_3\mu_1)$ as
$x_3\rightarrow\infty$, and  by $exp (x_3\mu_1)$ as $x_3\rightarrow
-\infty$; for $E_{(0,0,0) (1,0,\infty)}$, it is growth bounded (up to
polynomial factors) by $exp (x_3\mu_1)$ as $x_3\rightarrow\infty$, and  by
$exp(x_3(\mu_0-\mu_1))$ as $x_3\rightarrow -\infty$. Let us consider the
case of  $E_{(0,1,0) (0,1,\infty)}$. Our bounds for the section
(\ref{extension}) tell us that it is a (holomorphic) has a global section of
$E_{(0,1,0) (0,1,\infty)}$ on the $\bbp^1$ above
$(\eta,\zeta) = (w,0)$.

The sheaves $E_{(p,q,0) (p',q',\infty)}$ have degree  $2(p+p'-2) +q+q'$ on
the fibers of the projections to $T\bbp1$, from our identifications of them
given above. In particular,
the degree of $E_{(0,1,0) (0,1,\infty)}$ is $-2$.
Now, starting from a point of the Nahm spectral curve, we
have produced a global section of $E_{(0,1,0) (0,1,\infty)}$ over the
corresponding Riemann sphere; the Riemann--Roch theorem then tells us that
the first cohomology is also non-zero, so that we are in the support of
$R^1\tilde\pi_*(E_{(0,1,0) (0,1,\infty)})$, which is exactly the caloron spectral curve
$S_1$. As both $S_1$ are of the same degree and unreduced, they are identical.

Proceeding similarly for the spectral Nahm spectral curve $S_0$, we
obtain a global section of $E_{(0,0,0) (1,0,\infty)}$ over the corresponding
twistor line, and so show that the line lies in the support of
$R^1\tilde\pi_*(E_{(0,0,0) (1,0,\infty)})$, which is exactly the caloron spectral curve $S_0$.

We have now identified the spectral curves for the solutions to Nahm's
equations and the caloron that it produces. We now note that, in addition to
the bundle $V$ over $\SR$ that Nahm's construction produces,
there is another bundle, $\hat V$, obtained by taking the spectral data
associated to the solution of Nahm's equations, and feeding it in to the
reconstruction of a caloron from its spectral data.
Indeed, from a solution to Nahm's equations, we have
seen that we can can define sheaves ${\cal L}_z$; the fact that they
satisfy the boundary conditions  by Proposition \ref{nahm-boundary} tells us
we can reconstruct a bundle $F$ of infinite rank over $T\PP^1$ using
Sequence (\ref{sequence}). We now identify the bundles $V$ and $\hat V$.

Starting from a caloron, we obtained $F$ as the pushdown from the
twistor space ${\cal T}$ of a rank 2 bundle $E$. If $w$ is a fibre
coordinate on ${\cal T}\rightarrow T\bbp^1$, we saw that it induced an
automorphism $W$ of $F$, such that, at $ w= w_0$, $E_{w_0}\simeq
F/\Im(W-w_0)$. In Sequence
(\ref{sequence}), the shift map $W$ identifies the $n$th entry in the
middle column with the $n+2$th
(e.g. $L^{(p-1)\mu_0+\mu_1}(2k+j)\otimes{\cal I}_{S_{01}}$ with
$L^{p\mu_0-\mu_1}(2k+j)\otimes{\cal I}_{S_{10}}$) and similarly in the
right hand column. If $w = w(\zeta)$ is the equation of a twistor line
$L$ in ${\cal T}$, then quotienting by $W-w(\zeta)$ in
Sequence (\ref{sequence})
gives for $E$ over the line (recalling that $L^\mu_0$ is trivial over
the line)
\begin{equation}\label{sequence-E}
\begin{matrix}
&& L^{\mu_1}(2k+j)\otimes{\cal I}_{S_{01}}&&L^{\mu_1}(2k+j)[-S_{01}]|_{S_0}&&\\
E&\hookrightarrow&\oplus&\rightarrow&\oplus&\rightarrow&0.\\
&&L^{-\mu_1}(2k+j)\otimes{\cal I}_{S_{10}}&&L^{ -\mu_1}(2k+j)[-S_{10}]|_{S_1}\\
\end{matrix}
\end{equation}

{}From the twistor point of view, the space of global sections
of $E$ along the real line $L_x$ in $\mathcal{T}$
correspond to the fiber of the caloron bundle $\hat V$ over
the corresponding $x\in\SR$.
On the other hand, from the
Nahm point of view, $V_x$ correspond to
$\ker D_x^*$, and so to the cohomology
$\ker(D_2)/\Im(D_1)$ in Complex (\ref{Dirac-complex}).

To identify  $V$ and $\hat V$, we follow
closely \cite[pp. 80--84]{hurtubiseMurray}, so we
simply summarise the ideas.  The identification is made for the $x$ whose lines
$L_x$ which do not intersect  $S_{0}\cap S_{1}$;
to do it on this dense set is sufficient, as one is also identifying the connections.

Let us denote the sequence of sections
corresponding to (\ref{sequence-E}) by
\begin{equation}
0\rightarrow H^0(L_x,E)\rightarrow U_{\mu_1}\oplus U_{-\mu_1}\rightarrow
W_0\oplus W_1.
\end{equation}
What \cite{hurtubiseMurray} does is to identify $W_0$ and $W_1$ with 
solutions to $D_x^*(s) = 0$  on the interval $(\mu_1, \mu_0-\mu_1)$,
and $(-\mu_1,  \mu_1)$ respectively, and $U_{\mu_1}$ and $U_{-\mu_1}$
with $L^2$-solutions on a neighborhood of $\mu_1$ and $-\mu_1$
respectively.
The kernel $H^0(L_x,E)$ then
gets identified with global
$L^2$ solutions on the circle, which are elements of $V_x$.
Since from the twistor point of view
$\hat V_x\simeq H^0(L_x,E)$, we are done.
The case $j=0$ is
similar.

Having identified the bundles, one then wants to ensure that the
connections defined in both case are the same, and we do so again according to
\cite{hurtubiseMurray}. It  suffices to do
this identification along one null plane through the point $x$, as changing
coordinates will do the rest, and we can suppose that $x$ is
the origin, so $\eta=0$ in $T\PP^1$.
Let us choose the plane corresponding to $\zeta = 0$. From
the twistor point of view, parallel sections along this null plane
correspond to sections on the corresponding family of lines with fixed
values at $\zeta$ = 0. {}From the Nahm point of view, a parallel section
along this null plane is a family of
cohomology classes represented by cocycles $s=(s_1, s_2)$
in $\ker(D_2)$ that satisfy
$(\frac{\partial}{\partial x_1} +i\frac{\partial}{\partial x_2})s = 0,
(\frac{\partial}{\partial x_0} -i\frac{\partial}{\partial x_3})s=0$, modulo
coboundaries.
In particular,  the derivatives of $s_2$ are of the form
$(T_1+iT_2)\phi = A_0\phi$ for a suitable
$\phi$ depending on $x$. Following through the identifications of
\cite[pp. 80--84]{hurtubiseMurray} tells us that the derivatives 
$\nabla \hat s$
in suitable trivialisations of the corresponding
sections $\hat s$ on the family of lines are of the form 
$\nabla \hat s = A_0 \psi$
for the corresponding section $\psi$.
However, the defining relation for $A(\zeta)$ 
is $(\eta\bbi-A_0-\zeta A_1 -\zeta^2A_2) \phi = 0$. In
particular, that at $(\zeta, \eta) = 0$, $A_0\phi$ vanishes as a
function, and so $\hat s$ is parallel, as desired.

\subsection{From Nahm to caloron to Nahm to caloron.}
Let us first place ourselves in the generic situation.
We have seen first that the family of generic calorons maps continuously and
injectively into the family of generic spectral data; secondly, that generic
spectral data and generic solutions to Nahm's equations are equivalent;
thirdly, that starting from a caloron and producing a solution to Nahm's
equations then taking the spectral data of both objects gives the same
result, and fourthly  that the generic solution to Nahm's equation gives the
same caloron, whether you pass through the Nahm transform or through the
spectral data via the twistor construction. Taken together, those facts tell us
that the map from caloron to spectral data is a bijection,
that all three sets of data are equivalent, and that the six transforms
between them are pairwise inverses of each other.

Let us now leave the generic set. If we now take an arbitrary solution $T$
to Nahm's equations,
again satisfying all the conditions, we can fit it
into a continuous family $T(t)$ with $T(0) = T$ with $T(t)$ generic
for $t\ne0$; that it can be done follows from our description of moduli
given in Section \ref{moduli}.
The Nahm transform of this family is a continuous family $C(t)$ of calorons,
which in turn produces a continuous family of solutions $\tilde T(t)$ to
Nahm's equations. We do not know a priori whether the
boundary and symmetry conditions  are satisfied at $t=0$. However, for $t\ne
0$, $\tilde T(t) = T(t)$, and so $\tilde T(0) = T(0)$, and we are
done; the transform Nahm to caloron to Nahm is the identity.

If the caloron $C$ lies in the closure of the set of generic
calorons (presumably all calorons do, but it needs to be
proved, perhaps using the methods of Taubes
(\cite{TaubesPathConnectedYM}) to show that the moduli space is
connected), we again fit it as $C(0)$ into a family $C(t)$ of
calorons, with $C(t)$ generic for $t\ne 0$, all of same charge.  The Nahm
transform produces a family $T(t)$ of solutions to Nahm's equations, with
$T(t)$ satisfying all the conditions for $t\ne 0$. We also get a
family $S(t)$ of spectral data, which for all $t$ corresponds
to both $C(t)$ and $T(t)$.  We need to show that $T(0)$ satisfy
all the conditions.

Taking a limit, it is fairly clear that  the symmetry condition is
satisfied also for $t=0$.
 The boundary conditions are less obvious. From the small side,
 there is no problem, as the vanishing theorem \ref{vanishing-theorem}
 holds at $C(0)$, and so the solutions to Nahm's equations are
 continuous at the boundary of the interval.
{}From the large side, what saves us is the rigidity of representations
 into $SU(2)$; indeed,
the polar parts of the solutions are given by representations, and so
are fixed, in a suitable family of unitary gauges, hence preserved in a
limit. The  process, given in Section
\ref{sec:boundary-conditions}, of
passing to a constant gauge by solving $ds/dz +A_+(z) s = 0$ applies in the
limit near the singular points.  We have in the limit the same type of
transformations relating the basis of $K$ in which one has the solution to
Nahm's equations to  the continuous basis obtained from lifting up and
pushing down  as to obtain
Sequence (\ref{pushdown-sequence}), with the same process of producing endomorphisms
$\underline{A}(z,\zeta)(t)$, giving us again in the limit a continuous
endomorphism of a bundle $V$ over $T\bbp^1\times S^1$, with $V$
in fact lifted from $\bbp^1\times S^1$. The restrictions $V(\pm\mu_1)$
(on the large side) are of fixed type $V^1_{\mu_1} = {\cal
  O}^k\oplus V_{j-1,-j+1}$ for all $t$. The summand
$V_{j-1,-j+1}$ corresponds to the subsheaf ${\cal O}(j)$ of the sheaf
$F/(F^0_{p,1}+F^\infty_{-p,1})\otimes L^{\pm\mu_1}$, which exists in
the limit.  The polar part $\mathcal{C}(p(\zeta, t))$ of
Equation (\ref{normal-form-schem}), mapping $V_{j-1,-j+1}$ to itself,  has
limit $\mathcal{C}(p(\zeta, 0))$. By Equation
(\ref{determinant}), this latter limit is
determined by the spectral curves, and so is well defined. The other
summand $\OO^k$ is also well defined in the limit, as it is the
piece transferred from the small side, which is still ${\cal O}^k$
in the limit because of the vanishing theorem. The off-diagonal
vanishing as one goes back to the trivialisation for Nahm's equations  is
simply a consequence of the polar behaviour of solutions to $ds/dz +A_+(z) s
= 0$, and of the normal form for $\underline{A}(z,\zeta)(t)$ as in
Equation (\ref{normal-form}). In short, the
limit has exactly the same normal form, and so the same boundary
behaviour.

Finally, there is the question of irreducibility. The
reducible solutions correspond to calorons for which charge has
bubbled off; as our family has constant charge, this is precluded. The
limit solutions satisfies all the conditions necessary to produce a
caloron by the opposite Nahm transform; as the transform is involutive
on the generic member of the family, it is also involutive in the
limit,  and so the circle closes.

In short one has the following theorem.

\begin{theorem}
A) There is an equivalence between
\begin{enumerate}
\item  Generic calorons of charge $(k,j)$;
\item Generic solutions to Nahm's equations  on the circle satisfying
the conditions of Section \ref{sec:gaugefields}.B;
\item Generic spectral data.
\end{enumerate}

The equivalences of 1. and 2. are given in both
directions by the  Nahm transform.

B) The Nahm transforms give equivalences between
\begin{enumerate}
\item  Calorons of charge $(k,j)$, in the closure of the generic set;
\item Solutions to Nahm's equations  on the circle satisfying the
conditions of Section \ref{sec:gaugefields}.B.
\end{enumerate}
 \end{theorem}

\section{Moduli}\label{moduli}
The equivalence exhibited above allows us to classify calorons by
classifying appropriate solutions to Nahm's equations. One has to
guide us the example of monopoles, as classified in
\cite{DonaldsonSU2,hurtubiseClassification};
the (framed) monopoles for gauge group $G$, with
symmetry breaking to a torus  at infinity, charge $k$ are classified
by the space $Rat_k(\bbp^1, G_\bbc/B)$ of based degree $k$ rational
maps from  the Riemann sphere into the flag manifold $G_\bbc/B$. As
our calorons are Ka\v c--Moody monopoles, we should have the same theorem,
with framed calorons equivalent to rational maps from  $\bbp^1$ into
the loop group. Following an idea developed in \cite{Atiyah1984}, one
thinks of these as bundles on $\bbp^1\times\bbp^1$, with some extra
data of a flag along a line, and some framing.

The rank 2 bundle corresponding to an $SU(2)$ caloron, again following
the example of monopoles, should be the restriction of the bundle $E$
on $\tilde{\cal T}$ corresponding to the caloron, to the inverse image
of a point, say $\zeta = 0$, in $\bbp^1$. This inverse image is
$\bbp^1\times\bbc$; along the divisor $\{0\}\times\bbc$, the bundle
has the flag $E_0$, and along $\{\infty\}\times\bbc$, the bundle has
the flag $E_\infty$. We extend this bundle to infinity, using a framing,
in such a way that $E_\infty$ is a trivial subbundle,
and $E_0$ has degree $-j$. This bundle will be  corresponding to
the caloron.

The theorem, however, is proven in terms of solutions to Nahm's
equations; indeed, we show in \cite{benoitjacques1} that both
bundles and solutions to Nahm's equations satisfying the appropriate
conditions are describable in terms of a geometric quotient of a
family of matrices.

\begin{theorem}
Let $k\ge 1, j\ge 0$ be integers. There is an equivalence between

1) vector bundles $E$ of rank two on $\bbp^1\times \bbp^1$, with
   $c_1(E) = 0, c_2(E) =k$,
trivialised along $\bbp^1\times\{\infty\}\cup\{\infty\}\times\bbp^1$,
with a  flag $\phi\colon{\cal O}(-j)\hookrightarrow E$ of degree
$j$ along $\bbp^1\times\{0\}$ such that $\phi
(\infty) ({\cal O}(-j)) = {\rm span}(0,1)$;

2) framed, irreducible solutions to Nahm's equations
on the circle, with rank $k$ over $(\mu_1,\mu_0-\mu_1)$, rank $k+j$ over
$(-\mu_1, \mu_1)$, with the boundary conditions defined above,
modulo the action of the unitary gauge group;

3) complex matrices $A,B$ ($k\times k$), $C$ ($k\times 2$), $D$ ($2\times k$),
$A'$ ($j\times k$),
$B'$ ($1\times k$), $C'$ ($j\times 2$), satisfying the
monad equations
\begin{align*}
  [A,B] + CD&=0,                                           \\
  \begin{pmatrix}B'\\0\end{pmatrix}A + S(j,0)A' - A'B -C'D&=0,\\
  -e_+ A'+ \begin{pmatrix}1&0\end{pmatrix} D &=0,
\end{align*}
and the genericity conditions
\begin{align*}
\begin{pmatrix}A-y\\B-x\\D\end{pmatrix}
            &\text{ is injective for all }x,y\in \bbc,
\end{align*}      
\begin{align*}    
\begin{pmatrix}x-B&A-y& C\end{pmatrix}
            &\text{ is surjective for all }x,y\in \bbc,\\
\begin{pmatrix}x-B&A& C& 0\\ 
           \begin{pmatrix}-B'\\0\end{pmatrix}& A'& C'& x-S(j,0)\\
 0&0& \begin{pmatrix}1&0\end{pmatrix}& -e_+\end{pmatrix}
            &\text{ is surjective for all }x \in \bbc, \\
\begin{pmatrix}\begin{pmatrix}A\\
A'\end{pmatrix}&\begin{pmatrix}C_2\\ C'_2
       \end{pmatrix}&M\begin{pmatrix}C_2\\ C'_2\end{pmatrix}&\cdots&
            M^{j-1}\begin{pmatrix}C_2\\ C'_2\end{pmatrix}
       \end{pmatrix}&\text{ is an isomorphism,}
\end{align*}
 modulo the action of $Gl(k,\bbc)$ given by
\begin{equation*}
(A,B,C,D,A',B',C')\mapsto
        (gAg^{-1}, gBg^{-1}, gC, Dg^{-1},A'g^{-1}, B' g^{-1},C').
\end {equation*}
Here $C_i$, $C_i'$ are the $i$th column of $C$, $C'$, and
 $D_i$ the $i$th row of $D$, Equation (\ref{eqn:defS}) gives $S(j,0)$, and
$e_{+}:=\begin{pmatrix}0&\cdots &0 &1\end{pmatrix}$,
\begin{equation*}
M:=\begin{pmatrix}B& -C_1e_+\\ \begin{pmatrix} B'\\0\end{pmatrix}&S(j,0)-C_1'e_+\end{pmatrix}.
\end{equation*}
\end{theorem}

We see in \cite{benoitjacques1} that  $B$ and $M$ are
conjugate to the matrix $T_1+iT_2$ on the intervals
$(\mu_1,\mu_0-\mu_1)$ and $(-\mu_1, \mu_1)$ respectively.
Choosing $B$ diagonal, and $B$ and $M$
with distinct and disjoint eigenvalues, it is not difficult to build
solutions to the various matrix constraints. Hence the
spectral curves for the obtained solutions to Nahm's equations are reduced and
have no common components, as the spectral curve over $\zeta = 0$ is
the spectrum of $T_1+iT_2$. Thus generic generic solutions to Nahm's
equations, and so generic calorons, exist.


\def\cprime{$'$}

\where
\end{document}